\documentclass{SCAE}%SCAEOL for online version; SCAE for publication version; SCAES for the paper dedicated to somebody.
\numberwithin{equation}{section}
\usepackage{epstopdf}
\usepackage{fancyhdr}
\usepackage{graphicx}
\usepackage{hyperref}
\usepackage{mathtools}
\usepackage{multicol}
\usepackage{subfigure}

\allowdisplaybreaks

\renewcommand{\d}{\mathrm{d}}

\newcommand{\R}{\mathbb{R}}
\newcommand{\N}{\mathbb{N}}
\newcommand{\Z}{\mathbb{Z}}
\newcommand{\E}{\mathbb{E}}

\newcommand{\br}{\bm{r}}
\newcommand{\bu}{\bm{u}}
\newcommand{\bv}{\bm{v}}
\newcommand{\bet}{\bm{\eta}}
\newcommand{\daoshu}[2]{\dfrac{\d #1}{\d #2}}
\newcommand{\piandao}[2]{\dfrac{\partial #1}{\partial #2}}
\newcommand{\bianfen}[2]{\dfrac{\delta #1}{\delta #2}}

\begin{document}

%Basic Information
\Year{XXXX} %
\Month{XX}
\Vol{XX} %
\No{X} %
\BeginPage{1} %
\EndPage{XX} %
\AuthorMark{Li X {\it et al.}}
\ReceivedDay{XX}
\AcceptedDay{XX}
\PublishedOnlineDay{; published online XX}
\DOI{XX} % The author doesn't need fill in it.

% \title[short text for running head]{full title}{comments for title}
\title{An unconditionally energy stable finite difference scheme for a stochastic Cahn-Hilliard equation}{}

% \author[]{Full name}{footnote}
% Remark:  One \author for one author

\author[1]{LI Xiao}{}
\author[2]{QIAO ZhongHua}{Corresponding author}
\author[3]{ZHANG Hui}{}

\address[{\rm1}]{School of Mathematical Sciences, Beijing Normal University, Beijing {\rm 100875}, China;}
\address[{\rm2}]{Department of Applied Mathematics, The Hong Kong Polytechnic University, Hong Kong, China;}
\address[{\rm3}]{Laboratory of Mathematics and Complex Systems, Ministry of Education and\\
School of Mathematical Sciences,Beijing Normal University, Beijing {\rm 100875}, China}

\Emails{ lixiao1228@163.com, zqiao@polyu.edu.hk, hzhang@bnu.edu.cn}

\maketitle

%     Abstract is required.

{\begin{center}
\parbox{14.5cm}{\begin{abstract}
In this work, the MMC-TDGL equation, a stochastic Cahn-Hilliard equation is solved numerically by using
the finite difference method in combination with a convex splitting technique of the energy functional.
For the non-stochastic case, we develop an unconditionally energy stable difference scheme which is proved to be uniquely solvable.
For the stochastic case, by adopting the same splitting of the energy functional,
we construct a similar and uniquely solvable difference scheme with the discretized stochastic term.
The resulted schemes are nonlinear and solved by Newton iteration.
For the long time simulation, an adaptive time stepping strategy is developed based on both first- and second-order derivatives of the energy.
Numerical experiments are carried out to verify the energy stability,
the efficiency of the adaptive time stepping and the effect of the stochastic term.\vspace{-3mm}
\end{abstract}}\end{center}}

%  Keyword is required.

%  \subjclass is required.

%%%%%%%%%%%%%%%%%%%%%%%%%%%%%%%%%%%%%%%%%%%%%%%%%%%%%%%%%%%%
\renewcommand{\baselinestretch}{1.2}
\begin{center} \renewcommand{\arraystretch}{1.5}
{\begin{tabular}{lp{0.8\textwidth}} \hline \scriptsize
{\bf Citation:}\!\!\!\!&\scriptsize Li X, Qiao Z H, Zhang H. \makeatletter\@titlehead.
Sci China Math, XX, XX,
 %\@Year, \@Vol: \@BeginPage--\@EndPage,
 doi:~\@DOI\makeatother\vspace{1mm}
\\
\hline
\end{tabular}}\end{center}

%%%%%%%%%%%%%%%%%%%%%%%%%%%%%%%%%%%%%%%%%%%%%%%%%%%%%%%%%%%%
%% Text of article.
%%%%%%%%%%%%%%%%%%%%%%%%%%%%%%%%%%%%%%%%%%%%%%%%%%%%%%%%%%%%
%    Section headings
\baselineskip 11pt\parindent=10.8pt  \wuhao

\section{Introduction}

The process of phase transition has attracted many theoretical and experimental studies
in the field of small molecules or polymer mixture systems \cite{Dynamics1990}.
Among various methods for simulating phase transitions of a uniform thermodynamic system,
Cahn-Hilliard dynamics, advanced by Cahn and Hilliard \cite{CH1958} as a model for spinodal decomposition,
turns out to be one of the most suitable models \cite{Mesoscale2002}.
To take a random disturbance into account, a stochastic term is usually added into the equation.
This stochastic equation, called the Cahn-Hilliard-Cook model \cite{CH1958}, is of the form:
\begin{equation}
\label{sec1_MMC_TDGL}
\piandao{\phi}{t}=D\Delta\bianfen{U(\phi)}{\phi}+\varepsilon\xi(\br,t),\quad\br\in\Omega,\ t>0,
\end{equation}
where $\phi=\phi(\br,t)$ is a conserved field variable representing the concentration of one of the components of the mixture
(or sometimes, the difference between the concentration of the two components of a binary mixture),
$U=U(\phi)$ is a coarse-grained free energy functional.
The parameter $D>0$ is the diffusion coefficient (assumed constant here) and $\varepsilon>0$ describes the strength of the noise.
The stochastic term $\xi$ is required to satisfy the fluctuation-dissipation theorem \cite{Dynamics1990,fluctuation1999}:
\begin{equation}
\label{sec1_fluctuation_dissipation}
\E[\xi(\br_1,t_1)\xi(\br_2,t_2)]=-2D\Delta\delta(\br_1-\br_2)\delta(t_1-t_2),
\end{equation}
where $\E$ represents the mathematical expectation operator.
Recently, the time-dependent Ginzburg-Landau (TDGL) mesoscopic simulation method, based on the Cahn-Hilliard-Cook model,
has been widely used to describe the phase transitions of mixtures of small molecules \cite{ChCo2002}
and mixtures with polymers and block copolymers \cite{SelfAssemble1994}.
Here we give a brief review on the applications of the stochastic Cahn-Hilliard equation (\ref{sec1_MMC_TDGL})
for the descriptions of the phase separation process.

For the small molecules or atomic systems,
the functional $U(\phi)$ is usually assumed to have the Ginzburg-Landau form \cite{Huang1998}:
$$U_{\text{GL}}(\phi)=\int_\Omega\Big(\frac{\mu}{4}\phi^4-\frac{\gamma}{2}\phi^2+\frac{\gamma^2}{4\mu}+\frac{\kappa}{2}|\nabla\phi|^2\Big)\,\d\br,$$
where $\mu,\gamma,\kappa$ are all positive constants.
This form of the functional $U$, combined with the equation (\ref{sec1_MMC_TDGL}),
is usually used to investigate the phase separation of the small molecule systems,
such as binary alloys, fluid mixtures, inorganic glasses \cite{Dynamics1990}.

For the studies of spinodal decomposition in polymer blend,
Flory and Huggins developed a lattice theory and gave the Flory-Huggins free energy \cite{Flory1953}.
Combined with the contribution made by de Gennes,
the Flory-Huggins-de Gennes free energy functional for mixture of two polymers is given by \cite{deGennes1980}
\begin{equation}
\label{sec1_Flory_Huggins_deGennes}
U_{\text{FH}}(\phi)=\int_\Omega\Big(f_{\text{FH}}(\phi)+\frac{1}{36\phi(1-\phi)}|\nabla\phi|^2\Big)\,\d\br,
\end{equation}
where $$f_{\text{FH}}(\phi)=\frac{\phi}{N_1}\ln\phi+\frac{1-\phi}{N_2}\ln(1-\phi)+\chi\phi(1-\phi)$$ is the Flory-Huggins free energy density,
$N_1,N_2$ represent the degree of polymerization of the polymer 1 and 2, respectively, and $\chi$ is the Huggins interaction parameter.
Combined with the Flory-Huggins-de Gennes free energy functional (\ref{sec1_Flory_Huggins_deGennes}),
the equation (\ref{sec1_MMC_TDGL}) can be used to describe the phase transition of the thermodynamic systems of the polymer mixtures,
such as a large class of hydrogels, a kind of network crosslinked by polymer chains.

Macromolecular microsphere composite (MMC) hydrogels, a kind of new hydrogels,
have become more and more popular in polymeric materials because of their high mechanical strength \cite{WangHuiliang2007}.
The microstructure of MMC hydrogels is composed of both polymer chains and macromolecular microspheres.
In \cite{LiXiao2014,ZhaiDan2013}, we have developed a reticular free energy to describe this kind of structure.
Replacing $f_{\text{FH}}$ by the reticular free energy in (\ref{sec1_Flory_Huggins_deGennes})
and combining the stochastic Cahn-Hilliard equation (\ref{sec1_MMC_TDGL}),
we obtain a new model, named by MMC-TDGL equation, which simulates the phase transition of the MMC hydrogels well.
The specific expression of the reticular free energy and the form of the MMC-TDGL equation will be given in section 2.

For the stochastic Cahn-Hilliard equation (\ref{sec1_MMC_TDGL}),
some studies consider the effect of the stochastic term that reflects the thermal disturbance driving the system to escape away from the initial state.
Many of them (see, e.g., \cite{Dynamics1990}) are devoted to considering the random effect in the initial condition,
while some works (see, e.g., \cite{ZhangWei2012}) focus on the minimal energy pathway
of a system from one metastable state to another driven by the stochastic term.
In \cite{LiXiao2014}, a stochastic function satisfying the fluctuation-dissipation theorem (\ref{sec1_fluctuation_dissipation}) was constructed and (\ref{sec1_MMC_TDGL}) was discretized by using a semi-implicit finite difference scheme,
that is, the linear part of $\phi$ in (\ref{sec1_MMC_TDGL}) is treated implicitly while the nonlinear part explicitly.
The discrete energy stability for the non-stochastic case was obtained with a remainder term with respect to the time step.
In other words, the strict energy stability could not be obtained, which motivates us to find some better numerical schemes.

Recently, energy stable schemes have attracted increasingly attention for phase field models,
see e.g. \cite{Furihata2001,TangTao2007,QiaoZhonghua2015,QiaoZhonghua2015_Xie,TangTao2006} and references therein.
As one of the energy stable methods, the convex splitting approach has been widely used in the simulation of
phase field models \cite{Wise2012JSC,QiaoZhonghua2014,LiuShen2015,Wise2009}.
The idea about the convex splitting was first proposed for gradient systems by Eyre \cite{Eyre1997}.
The fundamental observation is that the energy $F$ admits a splitting into convex and concave parts,
namely, $F=F_c-F_e$, where $F_c$ and $F_e$ are both convex.
Here we use the notation $F_c$ and $F_e$ in \cite{Eyre1997},
where $c$ refers to the contractive part of the energy and $e$ refers to the expansive part.
A nonlinear stabilized splitting scheme for the Cahn-Hilliard equation was developed as an example in \cite{Eyre1997},
where the energy functional was split into convex and concave parts, with the convex part treated implicitly and the concave part explicitly.
The unconditionally gradient stability and unique solvability properties were pointed out by Eyre \cite{Eyre1998} but not proved.
Eyre's idea was generalized to the case of the Swift-Hohenberg equation and the phase field crystal equation \cite{Wise2009}.
The convex splitting scheme is first-order convergent in time.
A rigorous and complete theory framework was developed, and the unconditionally energy stability and unique solvability were proved strictly therein.
Besides, the convex splitting method has been used to solve the equations of thin film epitaxy \cite{Wise2010}.
Recently, second-order in time and unconditionally energy stable schemes have been proposed \cite{Wise2012},
based on the convex splitting technique, and were solved by using a linear iteration algorithm \cite{Wise2014}.

In this paper, we work in a similar framework of the convex splitting method for the MMC-TDGL equation.
There are three main purposes of this work.
First, we present an unconditionally energy stable and uniquely solvable difference scheme for solving the MMC-TDGL equation without the stochastic term.
The scheme is based on a convex splitting of the energy functional.
To apply the framework of this method, a convex splitting of the energy functional with the specific reticular free energy should be found out.
We decompose the integrand as a difference of two parts and prove them to be both convex.
The unconditional energy stability and unique solvability of the derived difference scheme are then obtained.
Second, the stochastic term is taken into account with the same discretized form as it in \cite{LiXiao2014},
which leads to a similar difference scheme whose unique solvability is derived directly from the non-stochastic case.
Third, we develop an adaptive time stepping strategy which is efficient for long time simulations of the MMC-TDGL equation.
In \cite{QiaoZhonghua2011}, an adaptive time stepping method is proposed for the molecular beam epitaxy models, based on the variation of the energy.
Numerical results have demonstrated that the adaptive time stepping method can save much CPU time without losing the accuracy.
Their approach is to consider the derivative of the energy and to magnify it via a constant $\alpha$.
Because of the high nonlinearity of the energy in our model,
there are longer time intervals during which the energy varies slightly but the derivative varies faster before the sharp decay.
A large $\alpha$ must be used to capture the moment of the sharp decay accurately, which leads to the low computational efficiency.
In order to overcome this difficulty, we treat $\alpha$ as a function of the time instead of a constant.
After the sharp decay, the energy decreases more and more gently, so we should reduce $\alpha$ to obtain larger time steps.
This strategy allows us not only to capture the sharp decay of the energy more accurately, but also to use the time steps as large as possible.

Generally speaking,
one of the main difficulties for the simulations of the Cahn-Hilliard dynamics is to approximate the fourth-order derivatives.
One approach is to set $\psi=\delta U(\phi)/\delta\phi$ in \eqref{sec1_MMC_TDGL} and write the fourth-order problem \eqref{sec1_MMC_TDGL}
as a mix-form system consisting of two second-order equations with the unknowns $\phi$ and $\psi$.
In this approach, one just need to approximate the second-order derivatives.
This technique has been widely used to simulate the fourth-order problems,
e.g., the Cahn-Hilliard dynamics and the molecular beam epitaxy models,
combining with either finite difference methods \cite{SunZhizhong1995,Wise2010_2,Wise2012}
or finite element methods \cite{Elliott92,ZhangShuo2010,QiaoZhonghua2015_Xie}.
Another strategy is to discretize directly the fourth-order derivatives using the second-order accurate center-difference formula.
Numerical tests imply that the derived schemes will be satisfactory as long as the spatial mesh size is sufficiently small,
see, e.g., \cite{Furihata2001,Furihata2003} for the Cahn-Hilliard dynamics, \cite{QiaoZhonghua2011,Wise2010} for the molecular beam epitaxy models and 
\cite{Sun2015, Zhang2013} for the phase field crystal equation.
In addition, discretizing the fourth-order derivatives using a fourth-order compact difference scheme could be found in a recent work \cite{LiSun2012}.
In this paper, we adopt the second technique to deal with the fourth-order derivative.
We will present a group of numerical experiments to show how the numerical solutions depend on the mesh size,
and determine an appropriate mesh giving the reasonable results for the subsequent simulations.

The rest of the paper is organized as follows.
In section 2, we illustrate the mathematical model of the phase transition of the MMC hydrogel.
The difference schemes based on a convex splitting of the energy functional are presented in section 3.
We first give the strict proofs of the unconditionally energy stability and unique solvability of the scheme for the non-stochastic case.
Then we recall the construction of the stochastic term,
provide the scheme with the stochastic term and illustrate the unique solvability.
In section 4, Newton iteration combined with GMRES algorithm is used to solve the difference scheme numerically.
Besides, an adaptive time-stepping method is given here.
In section 5, some numerical experiments are conducted to verify the energy stability of the non-stochastic difference scheme
and to compare the stochastic scheme with that developed in the previous work.
Concluding remarks are given in section 6.

\section{The mathematical model: MMC-TDGL equation}

In this paper, we consider the MMC-TDGL equation (\ref{sec1_MMC_TDGL}) in two-dimensional space,
with assuming that $\Omega=(0,L_x)\times(0,L_y)$ and $\phi(\cdot,t)$ subject to the periodic boundary condition.
The solutions of (\ref{sec1_MMC_TDGL}) are minima of the energy functional
\begin{equation}
    \label{sec2_energyfunctional}
    U(\phi)=\int_\Omega\Big(S(\phi)+H(\phi)+\kappa(\phi)|\nabla\phi|^2\Big)\,\d\br,
\end{equation}
where $S(\phi)+H(\phi)$ is the reticular free energy density \cite{LiXiao2014,ZhaiDan2013}:
\begin{equation}
    \label{sec2_reticular_free}
    S(\phi)=\frac{\phi}{\tau}\ln\frac{\alpha\phi}{\tau}+\frac{\phi}{N}\ln\frac{\beta\phi}{\tau}+(1-\rho\phi)\ln(1-\rho\phi),
    \quad H(\phi)=\chi\phi(1-\rho\phi),
\end{equation}
and $\kappa(\phi)$ is the de Gennes coefficient \cite{deGennes1980}:
\begin{equation}
    \label{sec2_deGennes}
    \kappa(\phi)=\frac{1}{36\phi(1-\phi)}.
\end{equation}
Here we follow the notations defined in \cite{LiXiao2014,ZhaiDan2013}.
We denote by $\chi$ the Huggins interaction parameter,
by $N$ the degree of polymerization of the polymer chains,
and by $M$, which does not appear explicitly in (\ref{sec2_reticular_free}), the relative volume of one macromolecular microsphere.
The other numbers $\alpha,\beta,\tau$ and $\rho$ depend on $M$ and $N$ according to
$$\alpha=\pi\Big(\sqrt{\frac{M}{\pi}}+\frac{N}{2}\Big)^2,\quad\beta=\frac{\alpha}{\sqrt{\pi M}},\quad\tau=\sqrt{\pi M}N,\quad\rho=1+\frac{M}{\tau}.$$
All these parameters are positive.
Besides, $\rho$ is a little greater than one, and thus $\phi\in(0,1/\rho)\subset(0,1)$.
We will declare the values of $M,N$ and $\chi$ in section 5.

In \cite{LiXiao2014}, by using the fact that the derivative and expectation operations can be exchanged,
it is proved that expressing the noise term as $\xi=-\sqrt{2D}\,\nabla\cdot\bet$ leads to the relation (\ref{sec1_fluctuation_dissipation}),
where $\bet=\big(\eta_1(\br,t),\eta_2(\br,t)\big)^T$. Here $\eta_1$ and $\eta_2$ are independent two-dimensional space-time Gaussian white noises, satisfying
$\E[\eta_l(\br,t)]=0$ and $\E[\eta_l(\br_1,t_1)\eta_l(\br_2,t_2)]=\delta(\br_1-\br_2)\delta(t_1-t_2)$, $l=1,2$.
In the rest of the paper, we set $D=1$ for the normalization.

We note the energy non-increasing property of the Cahn-Hilliard equation without the noise term, which is mentioned in the introduction, i.e.,
$$\daoshu{}{t}U(\phi(t))\le 0,\quad t>0.$$
So we hope to develop a numerical scheme inheriting such a property strictly.

\section{Unconditionally energy stable difference scheme}

In essence, the convex splitting method is to approximate the original non-convex energy problem by a convex energy problem.
The foundation is a convex splitting of the non-convex energy, which turns out to be a sum of convex and concave parts.
In this section, we will give a convex splitting of the energy (\ref{sec2_energyfunctional}) and an unconditionally energy stable scheme based on it.
We concentrate mainly on the difference scheme in the non-stochastic case, namely, $\varepsilon=0$.
The similar difference scheme of the stochastic case will be illustrated in the last part of this section.

\subsection{Discrete-time, continuous-space scheme}

The following proposition is a preliminary to the existence of a convex splitting.

\begin{proposition}
\label{sec3_lemma_convex}
{\rm(1)} $S$ and $-H$ are both convex in $(0,1/\rho)$, where $S$ and $H$ are defined by {\rm(\ref{sec2_reticular_free})};

{\rm(2)} $K(u,v):=\kappa(u)v^2$ is convex in $(0,1/\rho)\times\R$, where $\kappa$ is defined by {\rm(\ref{sec2_deGennes})}.
\end{proposition}

\begin{proof}
For (1), differentiating $S$ and $H$ twice, we obtain
$$S''(\phi)=\Big(\frac{1}{\tau}+\frac{1}{N}\Big)\frac{1}{\phi}+\frac{\rho^2}{1-\rho\phi},\quad H''(\phi)=-2\chi\rho.$$
When $\phi\in(0,1/\rho)$, we have $S''(\phi)>0$ and $H''(\phi)<0$.

For (2), by some careful calculations, we obtain the Hessian of $K$:
\[
   \nabla^2K=
   \begin{pmatrix}
    \dfrac{(3u^2-3u+1)v^2}{18u^3(1-u)^3} & \dfrac{(2u-1)v}{18u^2(1-u)^2} \\
    \dfrac{(2u-1)v}{18u^2(1-u)^2} & \dfrac{1}{18u(1-u)} \\
   \end{pmatrix}.
\]
The first-order principal minors of the matrix $\nabla^2K$ are
$$D_1=\frac{(3u^2-3u+1)v^2}{18u^3(1-u)^3},\quad D_2=\frac{1}{18u(1-u)}.$$
The second-order principal minor is
$$D_{12}=\det(\nabla^2K)=\frac{(3u^2-3u+1)v^2-(2u-1)^2v^2}{18^2u^4(1-u)^4}=\frac{v^2}{18^2u^3(1-u)^3}.$$
These principal minors are all non-negative when $u\in(0,1/\rho)$ and $v\in\R$.
The Hessian $\nabla^2K$ is positive semi-definite and thus $K$ is convex in $(0,1/\rho)\times\R$.
\end{proof}

\begin{lemma}[Existence of a convex splitting]
\label{sec3_convex_split_exist}
Assume that $\phi:\Omega\to(0,1/\rho)$ is smooth enough.
Defining
\begin{equation}
    \label{sec3_energy_convex_split}
    U_c(\phi)=\int_\Omega\Big(S(\phi)+\kappa(\phi)|\nabla\phi|^2\Big)\,\d\br,\quad U_e(\phi)=-\int_\Omega H(\phi)\,\d\br,
\end{equation}
we have $U(\phi)=U_c(\phi)-U_e(\phi)$ with $U_c$ and $U_e$ both convex.
\end{lemma}

\begin{proof}
The splitting $U(\phi)=U_c(\phi)-U_e(\phi)$ is obvious.
We show the convexity of $U_c$ below.
Defining$$e_c(\bu)=S(u)+K(u,v)+K(u,w),\quad e_e(\bu)=-H(u),\quad\bu=(u,v,w)\in\R^3,$$
we have$$U_c(\phi)=\int_\Omega e_c(\phi,\phi_x,\phi_y)\,\d\br,\quad U_e(\phi)=\int_\Omega e_e(\phi,\phi_x,\phi_y)\,\d\br.$$
Proposition \ref{sec3_lemma_convex} suggests that $e_c$ and $e_e$ are convex in $(0,1/\rho)\times\R\times\R$,
namely, $\forall\bu,\bv\in(0,1/\rho)\times\R\times\R$, $\forall\lambda\in(0,1)$,
$$e_c(\lambda\bu+(1-\lambda)\bv)\le\lambda e_c(\bu)+(1-\lambda)e_c(\bv).$$
Setting $\bu=(\phi,\phi_x,\phi_y)$ and $\bv=(\psi,\psi_x,\psi_y)$ and integrating this inequality, we obtain
$$U_c(\lambda\phi+(1-\lambda)\psi)\le\lambda U_c(\phi)+(1-\lambda)U_c(\psi),$$
which suggests that $U_c$ is convex.
Similarly, we know that $U_e$ is convex, too.
\end{proof}

The following estimate is the foundation of the energy stability.
The proof, same as that of Theorem 1.1 in \cite{Wise2009}, is independent on the specific form of $U(\phi)$ and so is omitted.

\begin{lemma}
\label{sec3_estimate}
Assume that $\phi,\psi:\Omega\to\R$ are periodic and smooth enough.
If $U=U_c-U_e$ gives a convex splitting, then
\begin{equation}
    \label{sec3_convex_split_estimate}
    U(\phi)-U(\psi)\le(\delta_\phi U_c(\phi)-\delta_\phi U_e(\psi),\phi-\psi)_{L^2},
\end{equation}
where $\delta_\phi$ denotes the variational derivative.
\end{lemma}

Given a time step $s>0$.
For the MMC-TDGL equation with $\varepsilon=0$, we construct the discrete-time, continuous-space scheme of (\ref{sec1_MMC_TDGL}) as follow:
\begin{equation}
    \label{sec3_semi_discrete_scheme}
    \frac{\phi^{k+1}-\phi^k}{s}=\Delta\tilde{\mu},\quad\tilde{\mu}(\phi^{k+1},\phi^k):=\delta_\phi U_c(\phi^{k+1})-\delta_\phi U_e(\phi^k).
\end{equation}

\begin{theorem}
\label{sec3_semidiscrete_stability}
The scheme {\rm(\ref{sec3_semi_discrete_scheme})} is unconditionally energy stable,
meaning that for any time step $s>0$, we always have $$U(\phi^{k+1})\le U(\phi^k).$$
\end{theorem}

\begin{proof}
By choosing $\phi=\phi^{k+1}$ and $\psi=\phi^k$ in (\ref{sec3_convex_split_estimate}), we obtain
\begin{align*}
    U(\phi^{k+1})-U(\phi^k) & \le(\delta_\phi U_c(\phi^{k+1})-\delta_\phi U_e(\phi^k),\phi^{k+1}-\phi^k)_{L^2}
    =s(\tilde{\mu},\Delta\tilde{\mu})_{L^2}=-s\|\nabla\tilde{\mu}\|_{L^2}^2\le 0,
\end{align*}
which completes the proof.
\end{proof}

A theoretical analysis for the energy stability based on a convex splitting of the continuous energy is provided above.
For the numerical simulation, however, we need to work on a discretized space.
In the following several parts of this section, we will construct a discrete energy with a convex splitting,
and develop a fully discrete and unconditionally energy stable scheme corresponding to the discrete energy.

\subsection{Discretization of two-dimensional space}

Here we use the similar notations introduced in \cite{Wise2010_2,Wise2009}.
Let $h_x=L_x/m$, $h_y=L_y/n$, where $m,n\in\N$.
Define the $x$-direction mesh $x_i=(i-\frac{1}{2})h_x$, $i\in\Z$  with respect to $(0,L_x)$ and corresponding node sets
$$E_m=\{x_{i+\frac{1}{2}}\,|\,i=0,1,\dots,m\},\ C_m=\{x_i\,|\,i=1,2,\dots,m\},\ C_{\overline{m}}=\{x_i\,|\,i=0,1,\dots,m+1\}.$$
Similarly, we can define the $y$-direction mesh $y_j$ and node sets $E_n,C_n,C_{\overline{n}}$ with respect to $(0,L_y)$.
Define the function spaces
\begin{align*}
    \mathcal{C}_{m\times n}=\{\phi:C_m\times C_n\to\R\},\quad\mathcal{C}_{\overline{m}\times\overline{n}}=\{\phi:C_{\overline{m}}\times C_{\overline{n}}\to\R\},\\
    \mathcal{C}_{\overline{m}\times n}=\{\phi:C_{\overline{m}}\times C_n\to\R\},\quad\mathcal{C}_{m\times\overline{n}}=\{\phi:C_m\times C_{\overline{n}}\to\R\},\\
    \mathcal{E}_{m\times n}^{\mathrm{ew}}=\{f:E_m\times C_n\to\R\},\quad\mathcal{E}_{m\times n}^{\mathrm{ns}}=\{f:C_m\times E_n\to\R\}.
\end{align*}
The functions in $\mathcal{C}_{m\times n}$, $\mathcal{C}_{\overline{m}\times\overline{n}}$,
$\mathcal{C}_{\overline{m}\times n}$ and $\mathcal{C}_{m\times\overline{n}}$ are called cell-centered functions,
and are denoted by the Greek symbols $\phi,\psi$ with $\phi_{i,j}=\phi(x_i,y_j)$.
The functions in $\mathcal{E}_{m\times n}^{\mathrm{ew}}$ and $\mathcal{E}_{m\times n}^{\mathrm{ns}}$
are called east-west edge-centered functions and north-south edge-centered functions, respectively, and are denoted by the English symbols $f,g$.
For the east-west edge-centered function $f$, we let $f_{i+\frac{1}{2},j}=f(x_{i+\frac{1}{2}},y_j)$;
for the north-south edge-centered function $g$, we let $g_{i,j+\frac{1}{2}}=g(x_i,y_{j+\frac{1}{2}})$.
We say a cell-centered function $\phi\in\mathcal{C}_{\overline{m}\times\overline{n}}$ is periodic if and only if
\begin{align*}
    \phi_{m+1,j}=\phi_{1,j},\quad\phi_{0,j}=\phi_{m,j},\quad j=1,2,\dots,n,\\
    \phi_{i,n+1}=\phi_{i,1},\quad\phi_{i,0}=\phi_{i,n},\quad i=0,1,\dots,m+1.
\end{align*}

Now we define some operators on the function spaces.
The edge-to-center averages and differences,
$a_x,d_x:\mathcal{E}_{m\times n}^{\mathrm{ew}}\to\mathcal{C}_{m\times n}$
and $a_y,d_y:\mathcal{E}_{m\times n}^{\mathrm{ns}}\to\mathcal{C}_{m\times n}$;
the center-to-edge averages and differences,
$A_x,D_x:\mathcal{C}_{\overline{m}\times n}\to\mathcal{E}_{m\times n}^{\mathrm{ew}}$
and $A_y,D_y:\mathcal{C}_{m\times\overline{n}}\to\mathcal{E}_{m\times n}^{\mathrm{ns}}$;
and the two-dimensional discrete Laplacian, $\Delta_h:\mathcal{C}_{\overline{m}\times\overline{n}}\to\mathcal{C}_{m\times n}$, are defined componentwise by
\begin{align*}
& a_xf_{i,j}=\frac{1}{2}(f_{i+\frac{1}{2},j}+f_{i-\frac{1}{2},j}),\quad d_xf_{i,j}=\frac{1}{h_x}(f_{i+\frac{1}{2},j}-f_{i-\frac{1}{2},j}),
    \quad\substack{i=1,2,\dots,m\\j=1,2,\dots,n}~,\\
& a_yg_{i,j}=\frac{1}{2}(g_{i,j+\frac{1}{2}}+g_{i,j-\frac{1}{2}}),\quad d_yg_{i,j}=\frac{1}{h_y}(g_{i,j+\frac{1}{2}}-g_{i,j-\frac{1}{2}}),
    \quad\substack{i=1,2,\dots,m\\j=1,2,\dots,n}~,\\
& A_x\phi_{i+\frac{1}{2},j}=\frac{1}{2}(\phi_{i+1,j}+\phi_{i,j}),\quad D_x\phi_{i+\frac{1}{2},j}=\frac{1}{h_x}(\phi_{i+1,j}-\phi_{i,j}),
    \quad\substack{i=0,1,\dots,m\\j=1,2,\dots,n}~,\\
& A_y\phi_{i,j+\frac{1}{2}}=\frac{1}{2}(\phi_{i,j+1}+\phi_{i,j}),\quad D_y\phi_{i,j+\frac{1}{2}}=\frac{1}{h_y}(\phi_{i,j+1}-\phi_{i,j}),
    \quad\substack{i=1,2,\dots,m\\j=0,1,\dots,n}~,\\
& \Delta_h\phi_{i,j}=d_x(D_x\phi)_{i,j}+d_y(D_y\phi)_{i,j},\quad\substack{i=1,2,\dots,m\\j=1,2,\dots,n}~.
\end{align*}
The weighted inner-product $(\cdot,\cdot)_h$, $[\cdot,\cdot]_{\mathrm{ew}}$ and $[\cdot,\cdot]_{\mathrm{ns}}$ are defined as follow:
\begin{align*}
(\phi,\psi)_h & =h_xh_y\sum_{i=1}^m\sum_{j=1}^n\phi_{i,j}\psi_{i,j},
    \quad\phi,\psi\in\mathcal{C}_{m\times n},\\
[f,g]_{\mathrm{ew}} & =h_xh_y\sum_{i=1}^m\sum_{j=1}^na_x(fg)_{i,j},
    \quad f,g\in\mathcal{E}_{m\times n}^{\mathrm{ew}},\\
[f,g]_{\mathrm{ns}} & =h_xh_y\sum_{i=1}^m\sum_{j=1}^na_y(fg)_{i,j},
    \quad f,g\in\mathcal{E}_{m\times n}^{\mathrm{ns}}.
\end{align*}
The following proposition follows directly and the proof is omitted.

\begin{proposition}
\label{sec3_discrete_formula}
Assume that $\phi,\psi\in\mathcal{C}_{\overline{m}\times\overline{n}}$, $f\in\mathcal{E}_{m\times n}^{\mathrm{ew}}$,
$g\in\mathcal{E}_{m\times n}^{\mathrm{ns}}$ and $\phi,\psi$ are periodic, then

{\rm(1)}~$[f,A_x\phi]_{\mathrm{ew}}=(a_xf,\phi)_h$, $[f,D_x\phi]_{\mathrm{ew}}=-(d_xf,\phi)_h$;

{\rm(2)}~$[g,A_y\phi]_{\mathrm{ns}}=(a_yg,\phi)_h$, $[g,D_y\phi]_{\mathrm{ns}}=-(d_yg,\phi)_h$;

{\rm(3)}~$(\phi,\Delta_h\psi)_h=-[D_x\phi,D_x\psi]_{\mathrm{ew}}-[D_y\phi,D_y\psi]_{\mathrm{ns}}=(\Delta_h\phi,\psi)_h$.
\end{proposition}

\subsection{A convex splitting of the discrete energy}

With the preparation above, we turn to discuss the discrete energy and the fully discrete scheme.
Define the discrete energy $F:\mathcal{C}_{\overline{m}\times\overline{n}}\to\R$ as
\begin{equation}
\label{sec3_discrete_energy}
F(\phi)=h_xh_y\sum_{i=1}^m\sum_{j=1}^n\Big(S(\phi_{i,j})+H(\phi_{i,j})
+\kappa(\phi_{i,j})\big(a_x((D_x\phi)^2)_{i,j}+a_y((D_y\phi)^2)_{i,j}\big)\Big).
\end{equation}

\begin{lemma}[Existence of a convex splitting]
\label{sec3_discrete_convex_split_exist}
Assume that $\phi\in\mathcal{C}_{\overline{m}\times\overline{n}}$ is periodic.
Defining
\begin{align}
F_c(\phi) & =h_xh_y\sum_{i=1}^m\sum_{j=1}^n\Big(S(\phi_{i,j})
+\kappa(\phi_{i,j})\big(a_x((D_x\phi)^2)_{i,j}+a_y((D_y\phi)^2)_{i,j}\big)\Big),\label{sec3_discrete_energy_convex_split1}\\
F_e(\phi) & =-h_xh_y\sum_{i=1}^m\sum_{j=1}^nH(\phi_{i,j}),\label{sec3_discrete_energy_convex_split2}
\end{align}
we have $F(\phi)=F_c(\phi)-F_e(\phi)$ with $F_c$ and $F_e$ both convex.
\end{lemma}

\begin{proof}
Let $K(u,v)$ be the function defined in Proposition \ref{sec3_lemma_convex}~(2), then
\begin{align*}
F_c(\phi) & =h_xh_y\sum_{i=1}^m\sum_{j=1}^n\Big(S(\phi_{i,j})
+\frac{1}{2}K(\phi_{i,j},D_x\phi_{i+\frac{1}{2},j})+\frac{1}{2}K(\phi_{i,j},D_x\phi_{i-\frac{1}{2},j})\\
& \qquad\qquad\qquad\qquad +\frac{1}{2}K(\phi_{i,j},D_y\phi_{i,j+\frac{1}{2}})+\frac{1}{2}K(\phi_{i,j},D_y\phi_{i,j-\frac{1}{2}})\Big).
\end{align*}
From Proposition \ref{sec3_lemma_convex}, $F_c$ and $F_e$ are linear combinations of some convex functions, and thus convex.
\end{proof}

Here we give the expressions of the variational derivatives
$\delta_\phi F_c(\phi)$ and $\delta_\phi F_e(\phi)$ as follows:
\begin{align*}
\delta_\phi F_c(\phi)
& =S'(\phi)+\kappa'(\phi)\big(a_x((D_x\phi)^2)+a_y((D_y\phi)^2)\big)-2d_x\big(A_x\kappa(\phi)D_x\phi\big)-2d_y\big(A_y\kappa(\phi)D_y\phi\big),\\
\delta_\phi F_e(\phi) & =-H'(\phi).
\end{align*}

Now we describe the fully discrete scheme for the MMC-TDGL equation with $\varepsilon=0$.
The scheme is the following:
given $\phi^k\in\mathcal{C}_{\overline{m}\times\overline{n}}$ periodic, find $\phi^{k+1}\in\mathcal{C}_{\overline{m}\times\overline{n}}$ periodic such that
\begin{align}
\phi^{k+1} & -\phi^k=s\Delta_h\mu^{k+1},\label{sec3_full_discrete_scheme_1}\\
\mu^{k+1} & =\delta_\phi F_c(\phi^{k+1})-\delta_\phi F_e(\phi^k)\nonumber\\
    & =S'(\phi^{k+1})+\kappa'(\phi^{k+1})\big(a_x((D_x\phi^{k+1})^2)+a_y((D_y\phi^{k+1})^2)\big)\nonumber\\
    & \qquad -2d_x\big(A_x\kappa(\phi^{k+1})D_x\phi^{k+1}\big)-2d_y\big(A_y\kappa(\phi^{k+1})D_y\phi^{k+1}\big)+H'(\phi^{k}),\label{sec3_full_discrete_scheme_2}
\end{align}
where
$$S'(\phi)=\Big(\frac{1}{\tau}+\frac{1}{N}\Big)\ln\phi-\rho\ln(1-\rho\phi),\quad H'(\phi)=-2\chi\rho\phi,\quad\kappa'(\phi)=\frac{2\phi-1}{36\phi^2(1-\phi)^2}.$$
Since $\mu$ follows the Laplacian $\Delta_h$, we omit the constants in the expressions $S'(\phi)$ and $H'(\phi)$ above.

So far, we have developed the fully discrete scheme based on a convex splitting of the discrete energy (\ref{sec3_discrete_energy}).
The difference scheme is a system of nonlinear equations with respect to $\phi^{k+1}$, so we have to solve it iteratively.
Before solving it, we analyze the unique solvability and the energy stability of the scheme.

\subsection{Unconditional unique solvability}

Suppose $c\in\R$, and define $\mathcal{M}_c$ to be the whole functions in $\mathcal{C}_{\overline{m}\times\overline{n}}$ with $c$-mean, namely,
$$\mathcal{M}_c=\Big\{\phi\in\mathcal{C}_{\overline{m}\times\overline{n}}\,\Big|\,\frac{1}{mn}\sum_{i=1}^m\sum_{j=1}^n\phi_{i,j}=c\Big\}.$$
We just discuss on the zero-mean function space $\mathcal{M}_0$, because $\phi-c\in\mathcal{M}_0$ provided $\phi\in\mathcal{M}_c$.

\begin{lemma}
\label{sec3_lemma_unique_solvable}
For any $\phi\in\mathcal{M}_0$, there exists a unique periodic $\psi\in\mathcal{M}_0$ such that $-\Delta_h\psi=\phi$.
\end{lemma}

\begin{proof}
Let $L=-\Delta_h$.
We show that $L$ is symmetry and positive definite on the periodic zero-mean function space.
The symmetry is suggested by Proposition \ref{sec3_discrete_formula} (3).
Suppose $\psi\in\mathcal{M}_0$ is periodic, then
$$(L(\psi),\psi)_h=[D_x\psi,D_x\psi]_{\mathrm{ew}}+[D_y\psi,D_y\psi]_{\mathrm{ns}}\ge 0.$$
The equality is achieved only if $D_x\psi=0$ and $D_y\psi=0$ at every point, which suggests that $\psi$ is a constant function.
With the restriction that $\psi$ has zero mean, this constant must be zero, which proves that $L$ is positive definite.
\end{proof}

\begin{theorem}[Unique solvability]
\label{sec3_unique_solvable}
The difference scheme {\rm(\ref{sec3_full_discrete_scheme_1})-(\ref{sec3_full_discrete_scheme_2})} is uniquely solvable for any time step $s>0$.
\end{theorem}

\begin{proof}
From the proof of Lemma \ref{sec3_lemma_unique_solvable}, the operator $L:=-s\Delta_h$ is positive definite,
so it is nonsingular and $L^{-1}$ is also positive definite.
Define a functional on $\mathcal{M}_0$:
\begin{equation}
\label{sec3_unique_solvable_pf}
G(\phi)=\frac{1}{2}\big(L^{-1}(\phi),\phi\big)_h-\big(L^{-1}(\phi),\phi^k\big)_h+F_c(\phi)-\big(\phi,\delta_\phi F_e(\phi^k)\big)_h.
\end{equation}
With a little work, we can obtain its variational derivative
$$\delta_\phi G(\phi)=L^{-1}(\phi-\phi^k)+\delta_\phi F_c(\phi)-\delta_\phi F_e(\phi^k).$$
For any $\psi\in\mathcal{M}_0$, we have
$$\frac{\d^2G(\phi+\lambda\psi)}{\d\lambda^2}\Big|_{\lambda=0}=\big(L^{-1}(\psi),\psi\big)_h+\frac{\d^2F_c(\phi+\lambda\psi)}{\d\lambda^2}\Big|_{\lambda=0}.$$
The convexity of $F_c$ and the positive definiteness of $L^{-1}$ imply the convexity of $G$ on $\mathcal{M}_0$.

Supposing that $\phi^{k+1}$ solves the scheme (\ref{sec3_full_discrete_scheme_1})-(\ref{sec3_full_discrete_scheme_2}), we obtain
$$\phi^{k+1}-\phi^k+L\big(\delta_\phi F_c(\phi^{k+1})-\delta_\phi F_e(\phi^k)\big)=0,$$
namely,$$\delta_\phi G(\phi^{k+1})=L^{-1}(\phi^{k+1}-\phi^k)+\delta_\phi F_c(\phi^{k+1})-\delta_\phi F_e(\phi^k)=0,$$
which implies, $\phi^{k+1}$ solves (\ref{sec3_full_discrete_scheme_1})-(\ref{sec3_full_discrete_scheme_2}) if and only if $\delta_\phi G(\phi^{k+1})=0$.
The convexity of $G$ suggests that $G$ is minimized by $\phi^{k+1}$, which is unique.
\end{proof}

\subsection{Unconditional energy stability}

The following lemma is a discrete form of Lemma \ref{sec3_estimate}.

\begin{lemma}
Assume that $\phi,\psi\in\mathcal{C}_{\overline{m}\times\overline{n}}$ are periodic.
If $F=F_c-F_e$ gives a convex splitting, then
\begin{equation}
\label{sec3_discrete_convex_split_estimate}
F(\phi)-F(\psi)\le(\delta_\phi F_c(\phi)-\delta_\phi F_e(\psi),\phi-\psi)_h.
\end{equation}
\end{lemma}

\begin{proof}
Define a continuously differentiable function in $\R$ as $j_c(s)=F_c(\phi+s\psi)$.
With the convexity of $F_c$, we get the convexity of $j_c$ in $\R$, so $j_c(1)-j_c(0)\ge j_c'(0)(1-0)$,
that is, $F_c(\phi+\psi)-F_c(\phi)\ge(\delta_\phi F_c(\phi),\psi)_h$.
Replacing $\psi$ by $\psi-\phi$ leads to $$F_c(\psi)-F_c(\phi)\ge(\delta_\phi F_c(\phi),\psi-\phi)_h.$$
Repeating the deduction above on $F_e$ and exchanging $\phi$ and $\psi$ lead to $$F_e(\phi)-F_e(\psi)\ge(\delta_\phi F_e(\psi),\phi-\psi)_h.$$
Adding the two inequalities and multiplying by $-1$ give the result.
\end{proof}

\begin{theorem}[Energy stability]
\label{sec3_energy_stable}
The scheme {\rm(\ref{sec3_full_discrete_scheme_1})-(\ref{sec3_full_discrete_scheme_2})} is unconditionally energy stable,
meaning that for any time step $s>0$, we always have $$F(\phi^{k+1})\le F(\phi^k).$$
\end{theorem}

\begin{proof}
By choosing $\phi=\phi^{k+1}$ and $\psi=\phi^k$ in (\ref{sec3_discrete_convex_split_estimate}), we obtain
\begin{align*}
    F(\phi^{k+1})-F(\phi^k) & \le(\delta_\phi F_c(\phi^{k+1})-\delta_\phi F_e(\phi^k),\phi^{k+1}-\phi^k)_h\\
    & =s(\mu^{k+1},\Delta_h\mu^{k+1})_h=-s\big([D_x\mu^{k+1},D_x\mu^{k+1}]_\mathrm{ew}+[D_y\mu^{k+1},D_y\mu^{k+1}]_\mathrm{ns}\big)\le 0,
\end{align*}
where the last equality is given by Proposition \ref{sec3_discrete_formula}.
\end{proof}

\subsection{The difference scheme in the stochastic case}\label{sec3_stochastoc}

Before giving the difference scheme, we recall the discretization of the stochastic term
\begin{equation}
\label{stochastic_term}
\xi(\br,t)=-\sqrt{2}\,\nabla\cdot\bet.
\end{equation}
See \cite{LiXiao2014} in detail.

In the theory of the stochastic process (see, e.g., \cite{LiuCihua1980}), the space-time Gaussian white noise can be represented as $\eta=\d W/\d t$,
where $W(t)$ is the standard Wiener process on $L^2(\Omega)$.
Separating the variables of the Wiener process $W$, we obtain
$$W(t,x,y)=\sum_{p,q}\beta_{pq}(t)e_{pq}(x,y),\quad(x,y)\in\Omega,~t\ge 0,$$
where $\{e_{pq}\}$ is a set of normal orthogonal basis on $L^2(\Omega)$,
$\beta_{pq}(t)=(W(t),e_{pq})_{L^2(\Omega)}$, and $\{\beta_{pq}(t)\}$ is a sequence of independent Wiener process, thus
$$\frac{\beta_{pq}(t_{k+1})-\beta_{pq}(t_k)}{\sqrt{s}}\sim N(0,1).$$
Using the mid-rectangle quadrature formula, we approximate $\eta_{ij}^{k+\frac{1}{2}}$ as:
\begin{align*}
\eta_{ij}^{k+\frac{1}{2}}=\Big(\frac{\d W}{\d t}\Big)_{ij}^{k+\frac{1}{2}}
& \approx\frac{1}{h_xh_ys}
\int_{(i-\frac{1}{2})h_x}^{(i+\frac{1}{2})h_x}\int_{(j-\frac{1}{2})h_y}^{(j+\frac{1}{2})h_y}\int_{t_k}^{t_{k+1}}\frac{\d W}{\d t}\,\d x\d y\d t\\
& =\frac{1}{h_xh_ys}\sum_{p,q}(\beta_{pq}(t_{k+1})-\beta_{pq}(t_k))
\int_{(i-\frac{1}{2})h_x}^{(i+\frac{1}{2})h_x}\int_{(j-\frac{1}{2})h_y}^{(j+\frac{1}{2})h_y}e_{pq}\,\d x\d y.
\end{align*}
For $i=1,2,\cdots,m$ and $j=1,2,\cdots,n$, we choose
$$e_{ij}=\frac{1}{\sqrt{h_xh_y}}\bm{1}_{[(i-\frac{1}{2})h_x,(i+\frac{1}{2})h_x)\times[(j-\frac{1}{2})h_y,(j+\frac{1}{2})h_y)}.$$
For $p\not=i$ or $q\not=j$, by orthogonalizing the basis functions, we obtain
$$\int_{(i-\frac{1}{2})h_x}^{(i+\frac{1}{2})h_x}\int_{(j-\frac{1}{2})h_y}^{(j+\frac{1}{2})h_y}e_{pq}\,\d x\d y=0.$$
Then we have
$$\eta_{ij}^{k+\frac{1}{2}}\approx\frac{1}{h_xh_ys}(\beta_{ij}(t_{k+1})-\beta_{ij}(t_k))\sqrt{h_xh_y}
=\frac{1}{\sqrt{h_xh_ys}}\frac{\beta_{ij}(t_{k+1})-\beta_{ij}(t_k)}{\sqrt{s}}
=\frac{1}{\sqrt{h_xh_ys}}r_{ij}^{k},$$
where $\{r_{ij}^{k}\}$ is a sequence of standard normal random variables.
Therefore, the discretized form of the stochastic term (\ref{stochastic_term}) is
\begin{equation}
\label{discrete_stochastic_term}
\xi_{ij}^{k+\frac{1}{2}}=-\frac{\sqrt{2}}{\sqrt{h_xh_ys}}\Big(a_xD_x(r_1)_{ij}^{k}+a_yD_y(r_2)_{ij}^{k}\Big),
\end{equation}
where $(r_1)_{ij}^k$ and $(r_2)_{ij}^k$ are standard normal random variables.

Now we describe the difference scheme for the MMC-TDGL equation with $\varepsilon>0$ as follows:
given $\phi^k\in\mathcal{C}_{\overline{m}\times\overline{n}}$ periodic, find $\phi^{k+1}\in\mathcal{C}_{\overline{m}\times\overline{n}}$ periodic such that
\begin{equation}
\label{sec3_full_discrete_scheme_1_stoc}
\phi^{k+1}-\phi^k=s\Delta_h\mu^{k+1}+s\varepsilon\xi^{k+\frac{1}{2}},
\end{equation}
where $\mu^{k+1}$ is still expressed as (\ref{sec3_full_discrete_scheme_2}),
$\xi^{k+\frac{1}{2}}=\{\xi_{ij}^{k+\frac{1}{2}}:1\le i\le m,\ 1\le j\le n\}$ with components given by (\ref{discrete_stochastic_term}).
Since the stochastic term $\xi^{k+\frac{1}{2}}$ does not depend on the unknown $\phi^{k+1}$,
the unique solvability of (\ref{sec3_full_discrete_scheme_1_stoc}) is the direct corollary of Theorem \ref{sec3_unique_solvable}.

\begin{corollary}[Unique solvability]
\label{sec3_unique_solvable_stoc}
The difference scheme {\rm(\ref{sec3_full_discrete_scheme_1_stoc}),(\ref{sec3_full_discrete_scheme_2})} is uniquely solvable
for any time step $s>0$.
\end{corollary}

\section{Numerical methods for solving the difference scheme}

In the last section, we present an unconditionally energy stable difference scheme
(\ref{sec3_full_discrete_scheme_1})-(\ref{sec3_full_discrete_scheme_2}) for the non-stochastic MMC-TDGL equation.
It is obvious, as mentioned above, that the scheme is nonlinear with respect to the unknown $\phi^{k+1}$, and we have to solve it iteratively.
To start the iteration, we need to choose an appropriate initial value,
which may affect the convergence of the iteration, as well as, the limit of a convergent iteration.
Fortunately, the unique solvability (Theorem \ref{sec3_unique_solvable} or Corollary \ref{sec3_unique_solvable_stoc}) guarantees
that the limit is rightly the solution as long as the iteration converges whatever the initial value is.
Now the problem is to generate a convergent iteration.
The proof of Theorem \ref{sec3_unique_solvable} indicates that $\phi^{k+1}$ is the unique minimum of a convex functional $G$,
which motivates us to convert the problem that to solve the schemes
(\ref{sec3_full_discrete_scheme_1})-(\ref{sec3_full_discrete_scheme_2}) or
(\ref{sec3_full_discrete_scheme_1_stoc}),(\ref{sec3_full_discrete_scheme_2})
to a variation problem that to find the minimum of the objective functional $G$ using optimization methods.
Here, we adopt the standard Newton method to search the minimum and the GMRES method \cite{GMRES1986} to solve the Newton equation.
The optimization theories ensure the convergence of the Newton iterative procedure (see, e.g., \cite{NumericalOpt1999}).

\subsection{Newton-GMRES method}

For convenience, we view the functional $G$ as an $mn$-variable function.
The Newton-GMRES procedure for finding $\phi^{k+1}$ with given $\phi^{k}$ is as follows:
\begin{enumerate}
  \item Let $x^{(0)}=\phi^k$ and $x^{(l)}$ is the $l$-th iteration of $\phi^{k+1}$;
  \item Solving the Newton equation $\nabla^2G(x^{(l)})p^{(l)}=-\nabla G(x^{(l)})$ by GMRES method;
  \item Let $x^{(l+1)}=x^{(l)}+p^{(l)}$;
  \item If $\|p^{(l)}\|<\mathrm{tol}$, let $\phi^{k+1}=x^{(l+1)}$; otherwise, let $l=l+1$ and turn to step 2.
\end{enumerate}
According to the expression of $G$ and the positive definiteness of $L=-s\Delta_h$,
the Newton equation $\nabla^2G(\phi)\psi=-\nabla G(\phi)$ can be rewritten as
\begin{equation}
\label{sec4_newton_eq3}
\psi-s\Delta_h\big(H(\phi)\psi\big)
=-(\phi-\phi^k+s\varepsilon\xi^{k+\frac{1}{2}})+s\Delta_h\big(\delta_\phi F_c(\phi)-\delta_\phi F_e(\phi^k)\big),
\end{equation}
where $H(\phi)=\frac{1}{h_xh_y}\nabla^2F_c(\phi)$ and
\begin{align*}
H(\phi)\psi & =S''(\phi)\psi+\kappa''(\phi)\psi\big(a_x((D_x\phi)^2)+a_y((D_y\phi)^2)\big)+2\kappa'(\phi)\big(a_x(D_x\phi D_x\psi)+a_y(D_y\phi D_y\psi)\big)\\
& \qquad\qquad -2d_x\big(A_x(\kappa'(\phi)\psi)D_x\phi+A_x\kappa(\phi)D_x\psi\big)-2d_y\big(A_y(\kappa'(\phi)\psi)D_y\phi+A_y\kappa(\phi)D_y\psi\big).
\end{align*}
The sparse linear equations (\ref{sec4_newton_eq3}) with respect to the Newton step $\psi$ will be solved by the GMRES method.

\subsection{Adaptive time stepping technique}

For the problems with the property of long time energy decreasing, adaptive time stepping approaches are usually utilized to save the CPU time.
The basic idea for this technique is that a small time step will be used when the energy decays sharply and a large time step will be used otherwise.
One efficient way to adjust the time steps is given by \cite{QiaoZhonghua2011}:
\begin{equation}
\label{sec4_adapted_time}
s_{k+1}=\max\Big\{s_{\mathrm{min}},\frac{s_{\mathrm{max}}}{\sqrt{1+\alpha|U'(t_k)|^2}}\Big\},\quad\alpha=\mathrm{const.},
\end{equation}
where $U(t)=U(\phi(\cdot,t))$ is the energy defined by (\ref{sec2_energyfunctional}).
The smallest and largest time steps $s_{\mathrm{min}}$ and $s_{\mathrm{max}}$ give the lower and upper bound of the adaptive time steps, respectively,
namely, $s_{\mathrm{min}}\le s_{k+1}\le s_{\mathrm{max}}$.
The large $|U'(t_k)|$ leads to the small time step $s_{k+1}$, which corresponds to the sharp decay of the energy.
The constant $\alpha$ is often evaluated by $10^5$ or higher magnitude to control $s_{k+1}$ in a narrow range
so as to capture the moment when the energy decays sharply.
We use the non-stochastic case to illustrate the adaptive time stepping method.
The derivative of the energy can be calculate by numerically integrating the following:
\begin{equation}
\label{sec4_Uderivative}
U'(t_k)=\Big(\int_\Omega\bianfen{U}{\phi}\piandao{\phi}{t}\,\d\br\Big)\Big|_{t=t_k}
=-\Big(\int_\Omega\Big|\nabla\bianfen{U}{\phi}\Big|^2\,\d\br\Big)\Big|_{t=t_k}.
\end{equation}

Not only has MMC-TDGL equation the property of long time energy decreasing in the non-stochastic case,
but there are also some intervals when the energy decays gently.
In such intervals, we hope to use a little bigger time steps.
Even though $|U'(t)|$ is not very large now, it may be magnified by $\alpha$ so that $s_{k+1}$ is too small to accelerate the computation.
So we consider a variant of (\ref{sec4_adapted_time}), and especially, we treat $\alpha$ as a function with respect to $t_k$.
Our idea for the choice of $\alpha$ is as follows:
$\alpha$ is kept on a lower magnitude when $|U'|$ changes from a big value to a small one;
during this period, $|U'|$ is dominant for adjusting the time step.
When $|U'|$ starts to increase, which means the energy will decay more and more sharply, we raise $\alpha$ rapidly to a higher magnitude,
and keep $\alpha|U'|^2$ above a certain  magnitude; during this period, $\alpha$ is dominant.
Whether $|U'|$ increases or decreases is reflected via the sign of the second-order derivative $U''$.
More precisely, we define $\alpha_k$ as the following form:
\begin{equation}
    \label{sec4_adapted_time_alpha}
    \alpha_k=
    \begin{dcases}
        \alpha_{\mathrm{min}}, & U''(t_k)\ge 0,\\
        \alpha_{\mathrm{min}}-AU''(t_k), & U''(t_k)<0,
    \end{dcases}
\end{equation}
where $\alpha_{\mathrm{min}}$ is a lower bound of $\alpha_k$ and $A$ is a positive constant used to magnify the effect of $U''(t_k)$.
Here the second-order derivative $U''(t_k)$ is approximated by the backward difference
$$U''(t_k)\approx\frac{U'(t_k)-U'(t_{k-1})}{s_{k}},$$
and $U'(t_k)$ is calculated by (\ref{sec4_Uderivative}).

\section{Numerical experiments}

Our experiments are divided into two parts.
First, we use the difference scheme (\ref{sec3_full_discrete_scheme_1})-(\ref{sec3_full_discrete_scheme_2})
with the adaptive time stepping technique (\ref{sec4_adapted_time}),(\ref{sec4_adapted_time_alpha})
to solve the MMC-TDGL equation (\ref{sec1_MMC_TDGL}) in the non-stochastic case.
The results demonstrate the unconditional energy stability of the difference scheme and the advantage of the adaptive time stepping method.
Second, we use the difference scheme (\ref{sec3_full_discrete_scheme_1_stoc}),(\ref{sec3_full_discrete_scheme_2}) with the uniform time step
to simulate the MMC-TDGL equation (\ref{sec1_MMC_TDGL}) with various noise strengths.
We observe the effect of the stochastic term and compare the results here with those presented in \cite{LiXiao2014}.

We consider the domain $\Omega=(0,50)\times(0,50)$ and set $\chi=2.37$, $M=0.16$, $N=4.34$ in the model.
The tolerance of the Newton iteration method is set to be $10^{-9}$;
the tolerance and the restart times of the GMRES method for solving the Newton equation are $10^{-8}$ and $40$, respectively.

\subsection{The non-stochastic case}

The difference scheme (\ref{sec3_full_discrete_scheme_1})-(\ref{sec3_full_discrete_scheme_2}) is used to solve the MMC-TDGL equation with $\varepsilon=0$.
The initial condition is set to be $\phi(\br,0)=0.6+\zeta(\br)$,
where $\zeta(\br)$ is a random disturbance whose values range from $-0.15$ to $0.15$ on each mesh point $(x_i,y_j)$ of the $m\times n$ grid.
We first carry out a group of experiments to see how the numerical solutions depend on the spatial scales
so that we can choose an appropriate grid for the further simulations with fewer calculations and without losing the accuracy.

We fix the time step $s=0.001$ and use $m=n=50$, $100$, $200$ and $400$ to solve the equation, respectively.
The random disturbance $\zeta(\br)$ is evaluated on the $400\times400$ grid,
and then restricted on the corresponding nodes in the other coarse grids.
The numerical solutions at $t=5$ with $y=25$ fixed are plotted in Figure \ref{sec5_fig_grid}.
It is found from (a) that the curve with $m=n=200$ has obvious difference from those with $m=n=100$ and $50$,
while from (b), that it has almost no difference from those with $m=n=400$ except for tiny bias near $x=2$ and $x=23$.
By the comparison, we see the convergence as the spatial scale turns smaller,
and we are convinced that the numerical solutions on the $200\times200$ grid are reasonable.
The rest experiments will be carried out on the $200\times200$ grid.
Both the constant and adaptive time steps are adopted to obtain the numerical solutions.

\begin{figure}[!htp]
\centering
\subfigure[$m=n=50$, $100$ and $200$.]{\includegraphics[scale=0.45]{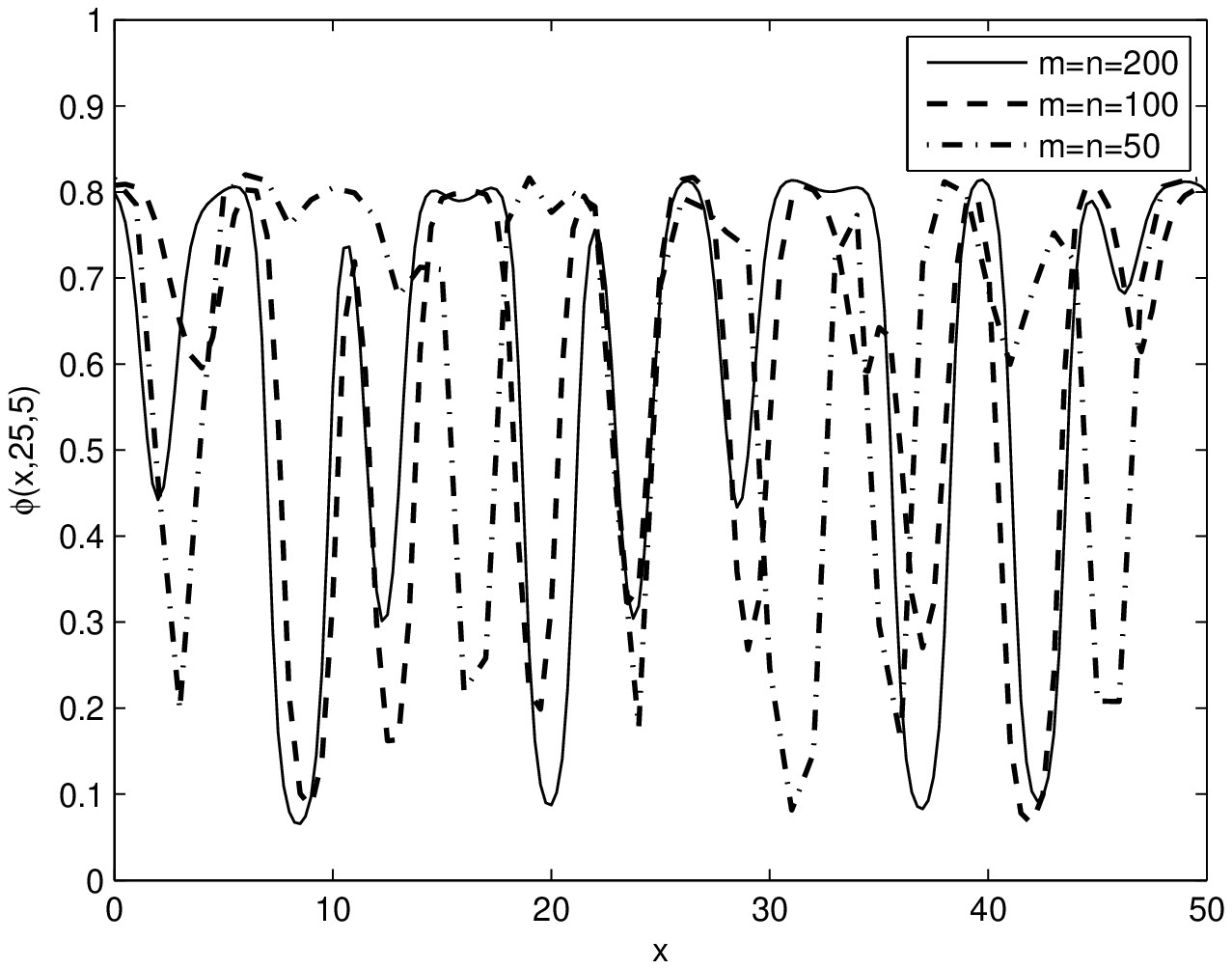}}
\subfigure[$m=n=200$ and $400$.]{\includegraphics[scale=0.45]{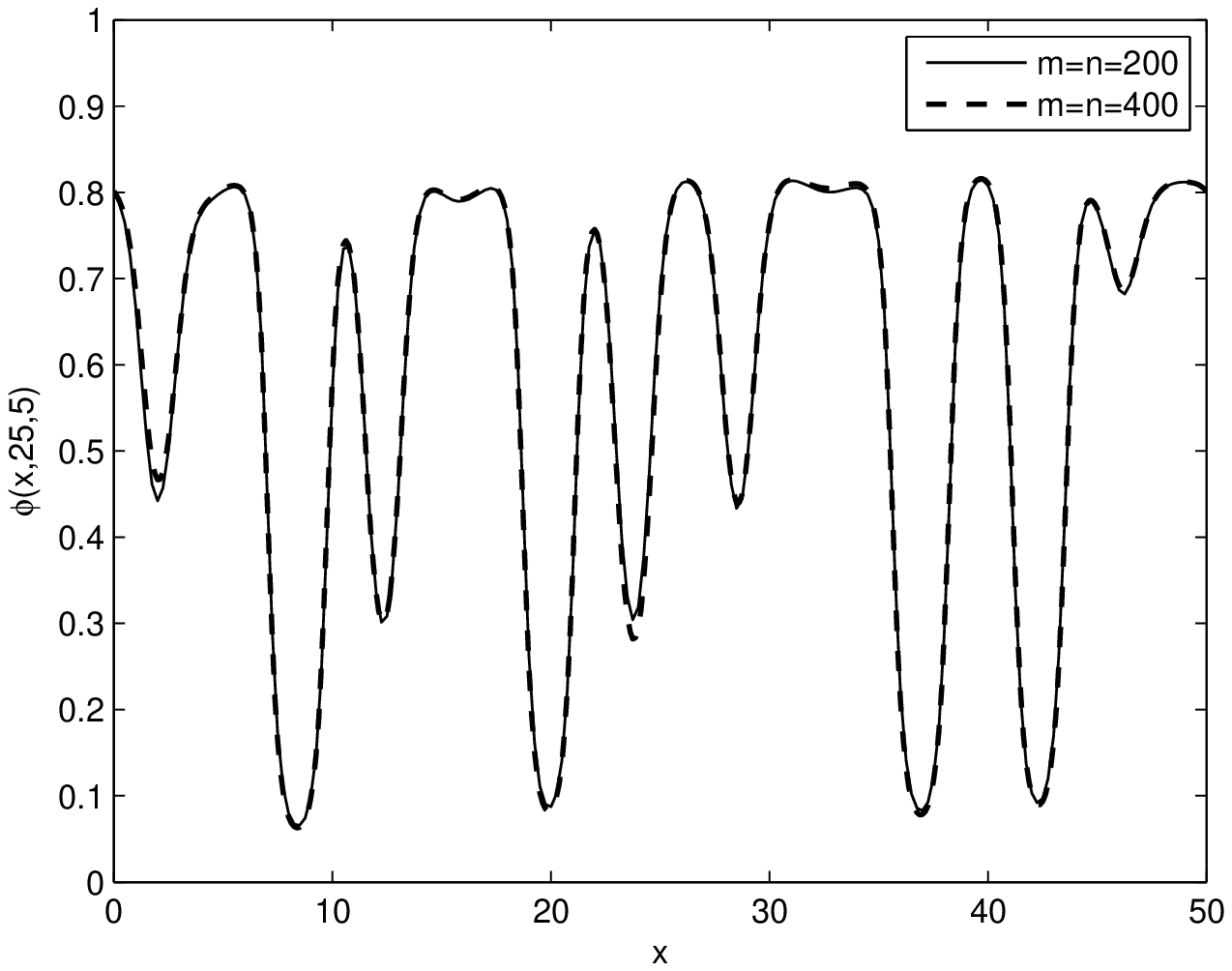}}
\caption{Numerical results obtained by the scheme
{\rm(\ref{sec3_full_discrete_scheme_1})-(\ref{sec3_full_discrete_scheme_2})} with various spatial scales.}
\label{sec5_fig_grid}
\end{figure}

We first use different constant time steps $s=0.0001$, $s=0.001$, $s=0.01$ and $s=0.1$ to solve the equation.
Figure \ref{sec5_fig_constant} shows the numerical results.
The numerical solutions at $t=20$ and $t=3$ are plotted in Figure \ref{sec5_fig_constant}(a) and (b), respectively.
It is found that the solutions with $s\le 0.001$ have few biases, while the solutions with $s\ge 0.01$ include large numerical errors.
The energy evolution is shown in Figure \ref{sec5_fig_constant}(c).
It is obvious that the curve corresponding to $s=0.1$ goes far from the others from about $t=2$.
Besides, we can see that other three curves match well after $t=6$,
but the simulations with larger time steps fail to capture the rapid transition of the phase (see Figure \ref{sec5_fig_constant}(d)).
Therefore, the time step $s=0.001$ is suitable for the simulation of the MMC-TDGL equation with $\varepsilon=0$.

\begin{figure}[h]
\centering
\subfigure[Solution at $t=20$ with $y=25$ fixed.]{\includegraphics[scale=0.45]{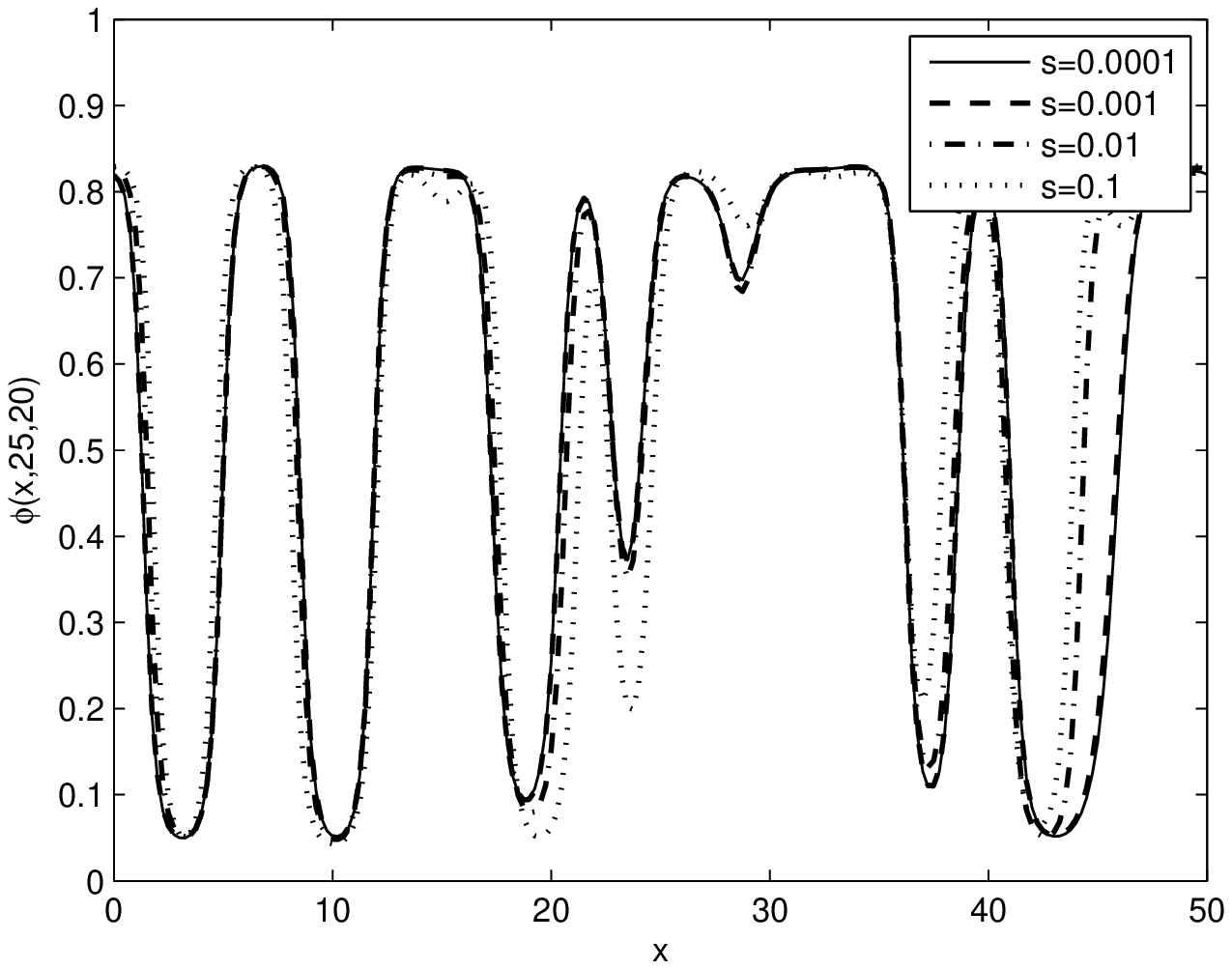}}
\subfigure[Solution at $t=3$ with $y=25$ fixed.]{\includegraphics[scale=0.45]{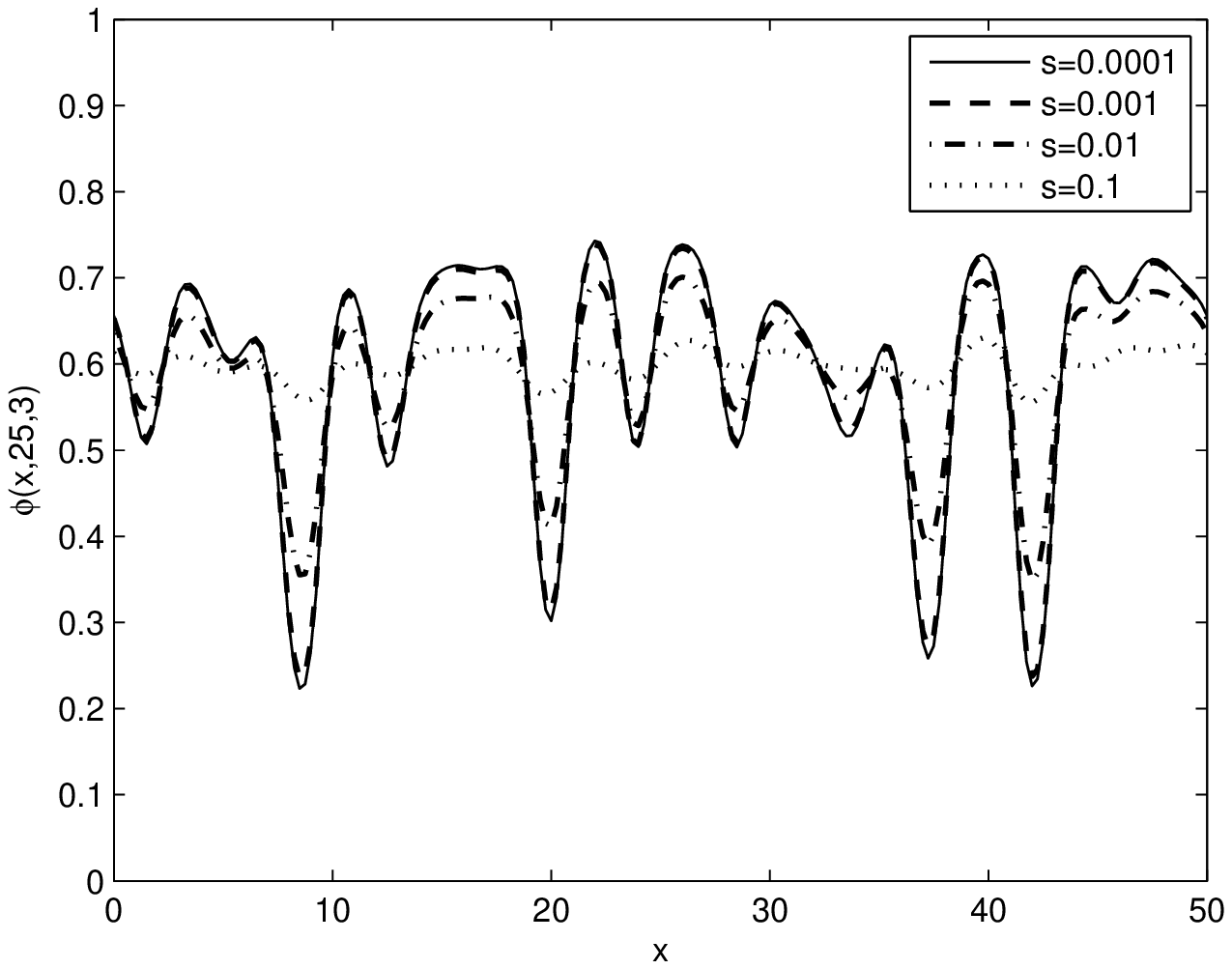}}
\subfigure[Energy evolution: $0\le t\le 20$.]{\includegraphics[scale=0.45]{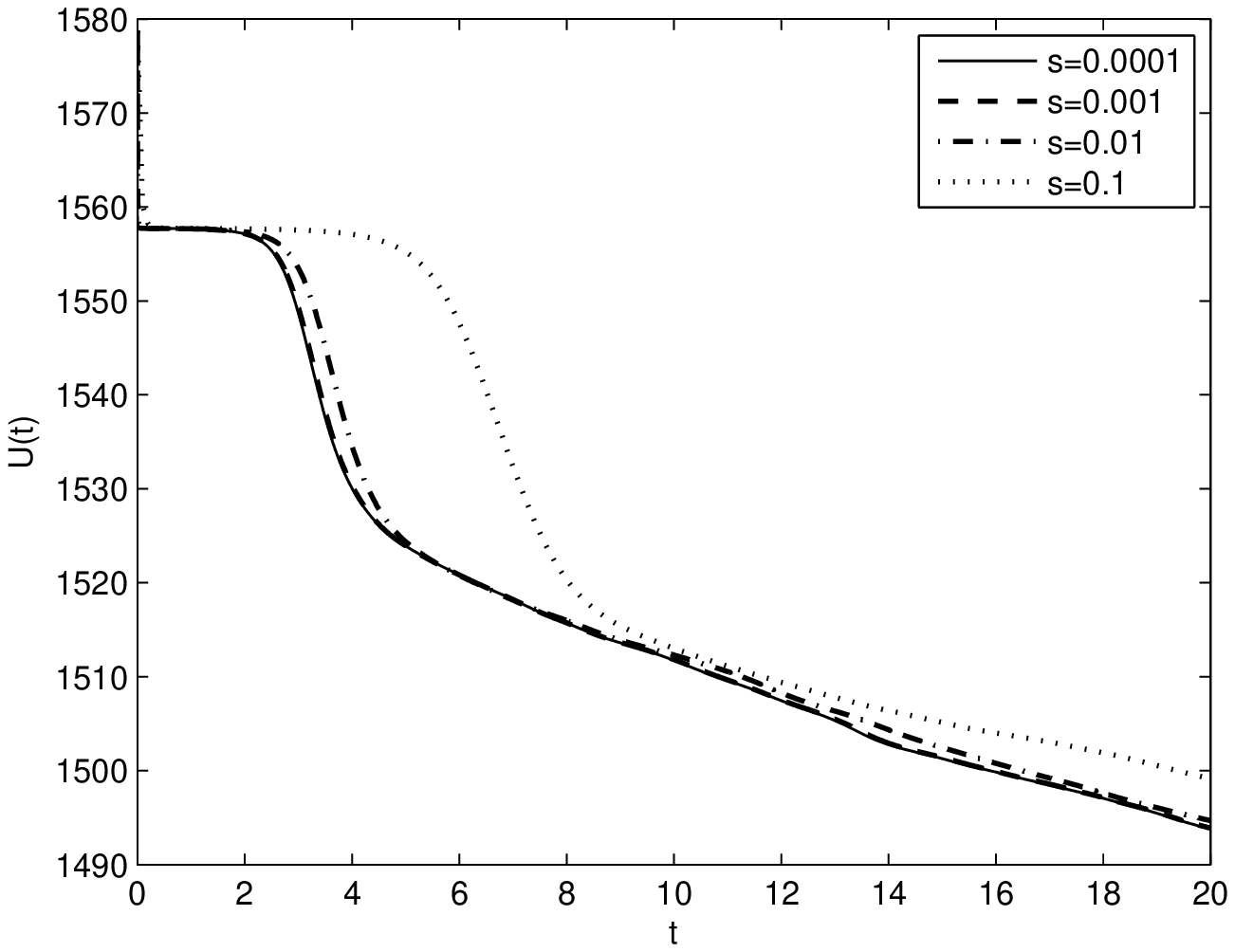}}
\subfigure[Energy evolution: $1.8\le t\le 5.5$.]{\includegraphics[scale=0.45]{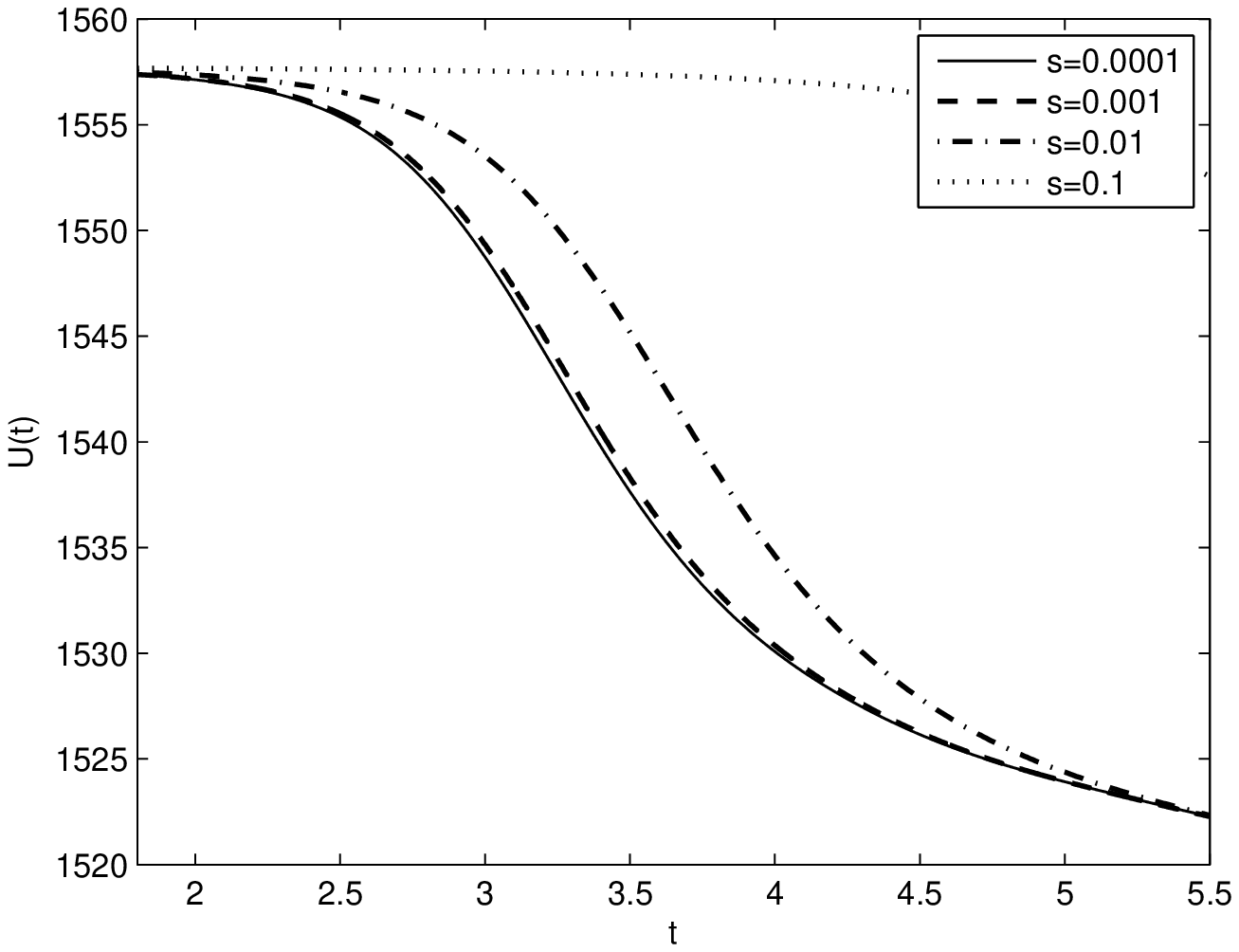}}
\caption{Numerical results obtained by the scheme
{\rm(\ref{sec3_full_discrete_scheme_1})-(\ref{sec3_full_discrete_scheme_2})} with the constant time steps.}
\label{sec5_fig_constant}
\end{figure}

For the adaptive time stepping technique, in order to enhance the efficiency as far as possible, we divide the evolution into two periods.
In the first period, from the beginning to the occurrence of the sharp decay of the energy,
the adaptive time step is calculated by (\ref{sec4_adapted_time}) with $\alpha$ replaced by $\alpha_k$ defined via (\ref{sec4_adapted_time_alpha}),
where we set $s_{\min}=0.001$, $s_{\max}=0.1$, $\alpha_{\mathrm{min}}=10^5$ and $A=10^6$.
We choose $\alpha_{\min}$ and $A$ large enough to capture the moment of the rapid decay.
In the second period, after the sharp decay of the energy, which we mark by $|U'|<3$,
we just set $\alpha_k\equiv 100$, with $s_{\min}$ and $s_{\max}$ holding, to loosen the restriction for the time steps to accelerate the computation.
Here we set the maximum time to be $T=20$
and compare the results obtained by the adaptive time steps with those obtained by the constant time step $s=0.001$.

\begin{figure}[h]
\centering
\subfigure[Solution at $t=20$ with $y=25$ fixed.]{\includegraphics[scale=0.45]{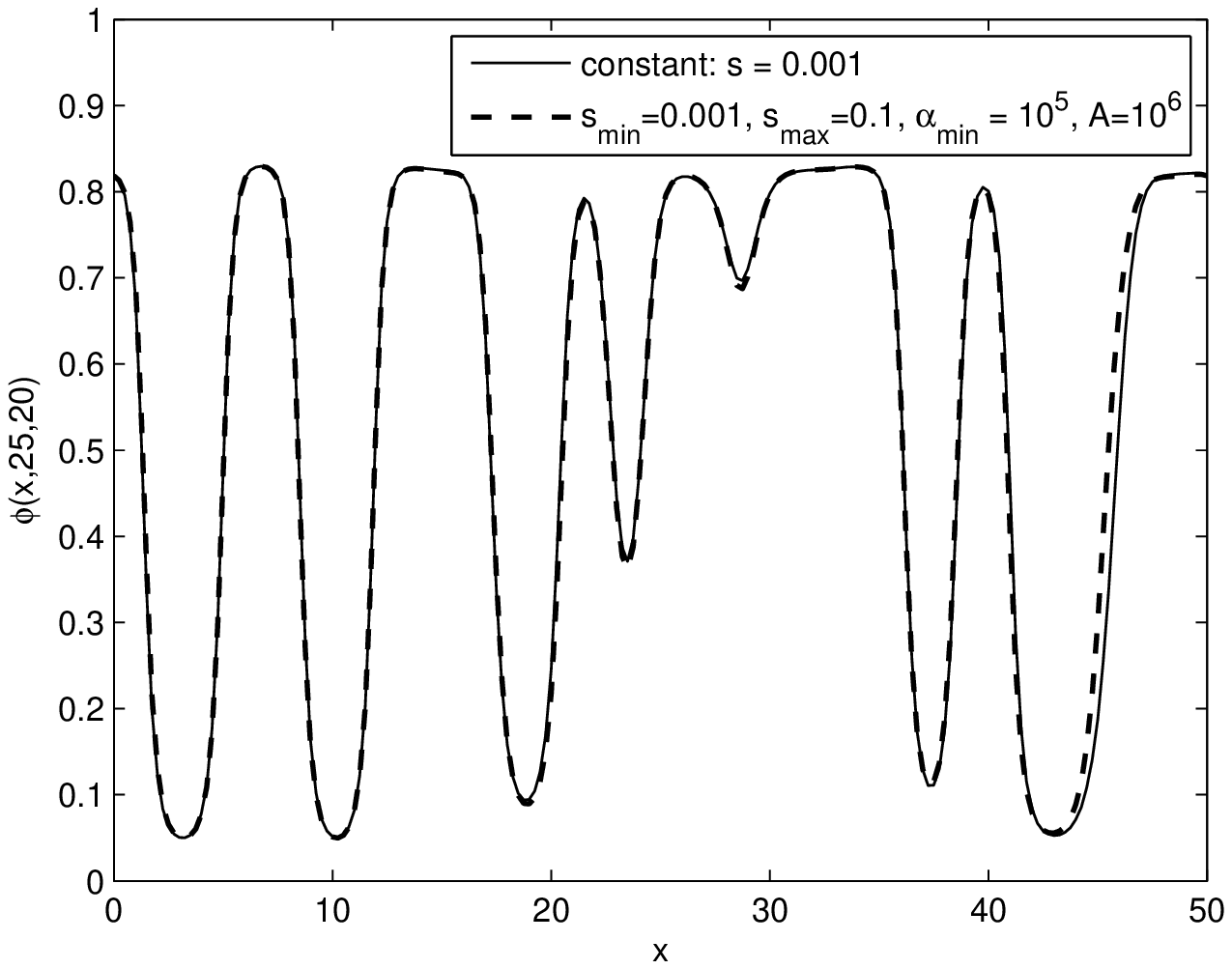}}
\subfigure[Solution at $t=3$ with $y=25$ fixed.]{\includegraphics[scale=0.45]{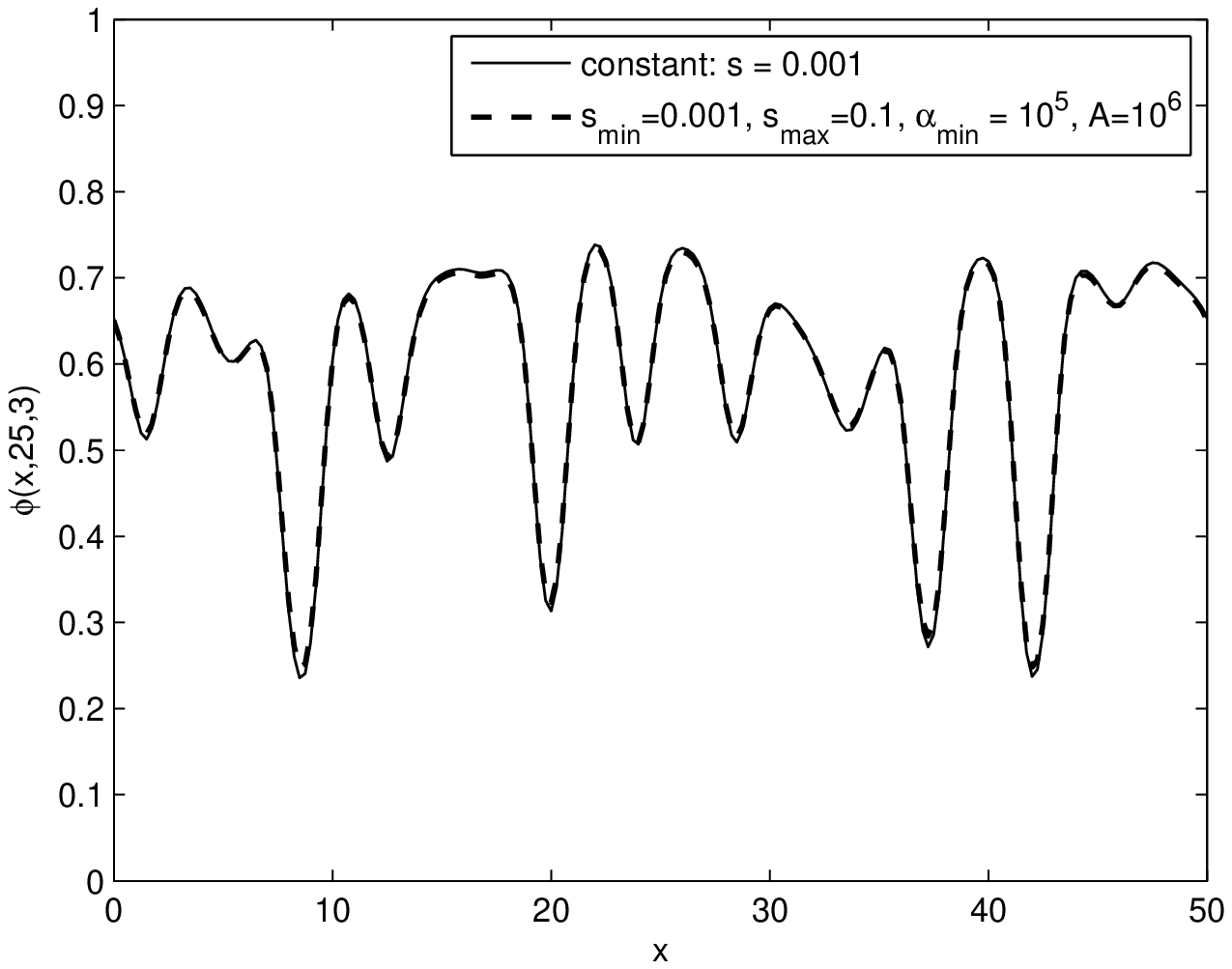}}
\subfigure[Energy evolution: $0\le t\le 20$.]{\includegraphics[scale=0.45]{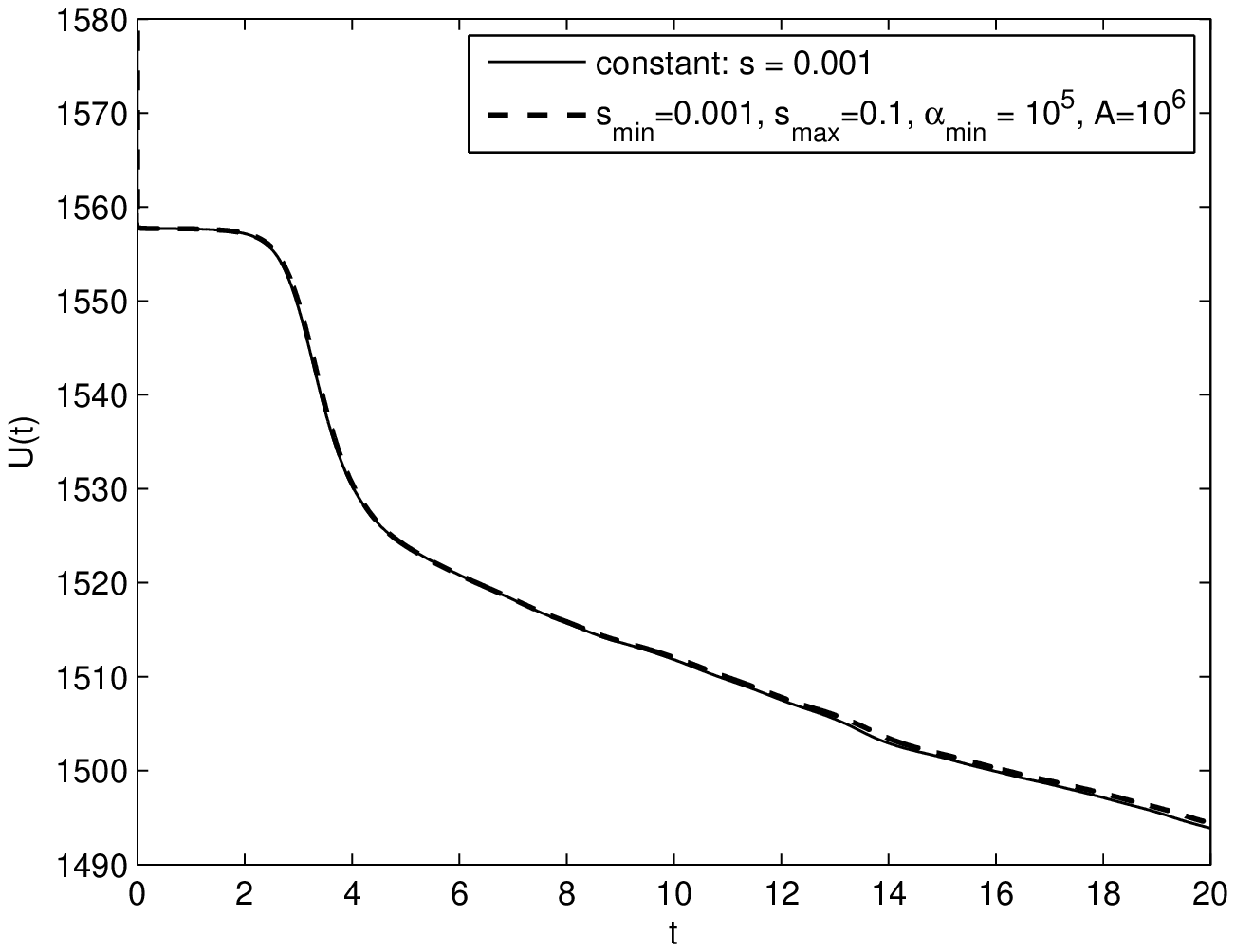}}
\subfigure[Energy evolution: $1.8\le t\le 5.5$.]{\includegraphics[scale=0.45]{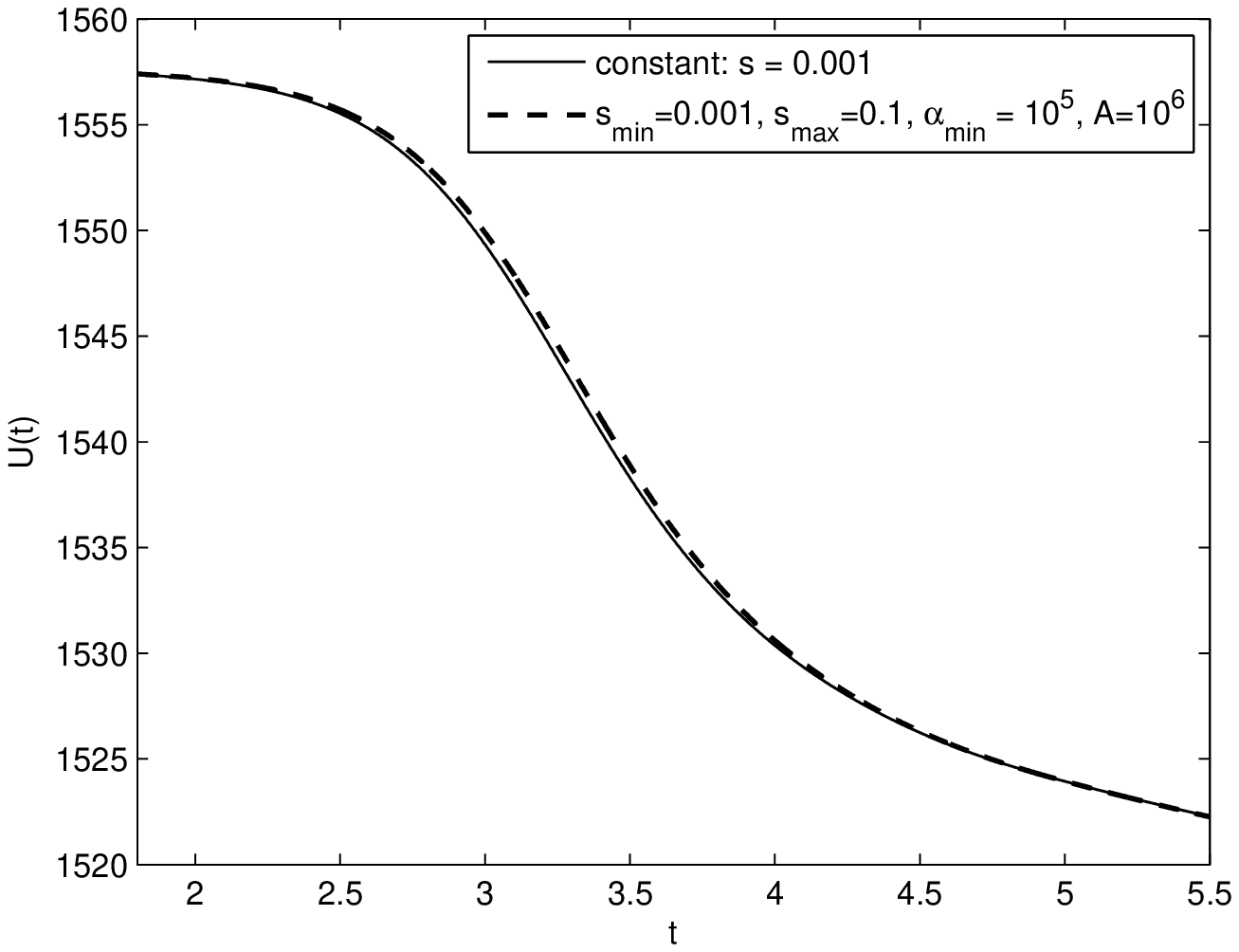}}
\caption{Numerical results obtained by the scheme
{\rm(\ref{sec3_full_discrete_scheme_1})-(\ref{sec3_full_discrete_scheme_2})} and adaptive time stepping method.}
\label{sec5_fig1_adaptive}
\end{figure}

Figure \ref{sec5_fig1_adaptive} shows the numerical results.
The numerical solutions at $t=20$ and $t=3$ are plotted in Figure \ref{sec5_fig1_adaptive}(a) and (b), respectively.
It is found that the adaptive time stepping technique preforms satisfactorily.
The energy evolution is shown in Figure \ref{sec5_fig1_adaptive}(c).
It is observed that the adaptive time stepping method captures the phase structure during the evolution,
especially for the phase transition stage where the energy decays sharply (see Figure \ref{sec5_fig1_adaptive}(b) and (d)).

\begin{figure}[!htp]
\centering
\includegraphics[scale=0.55]{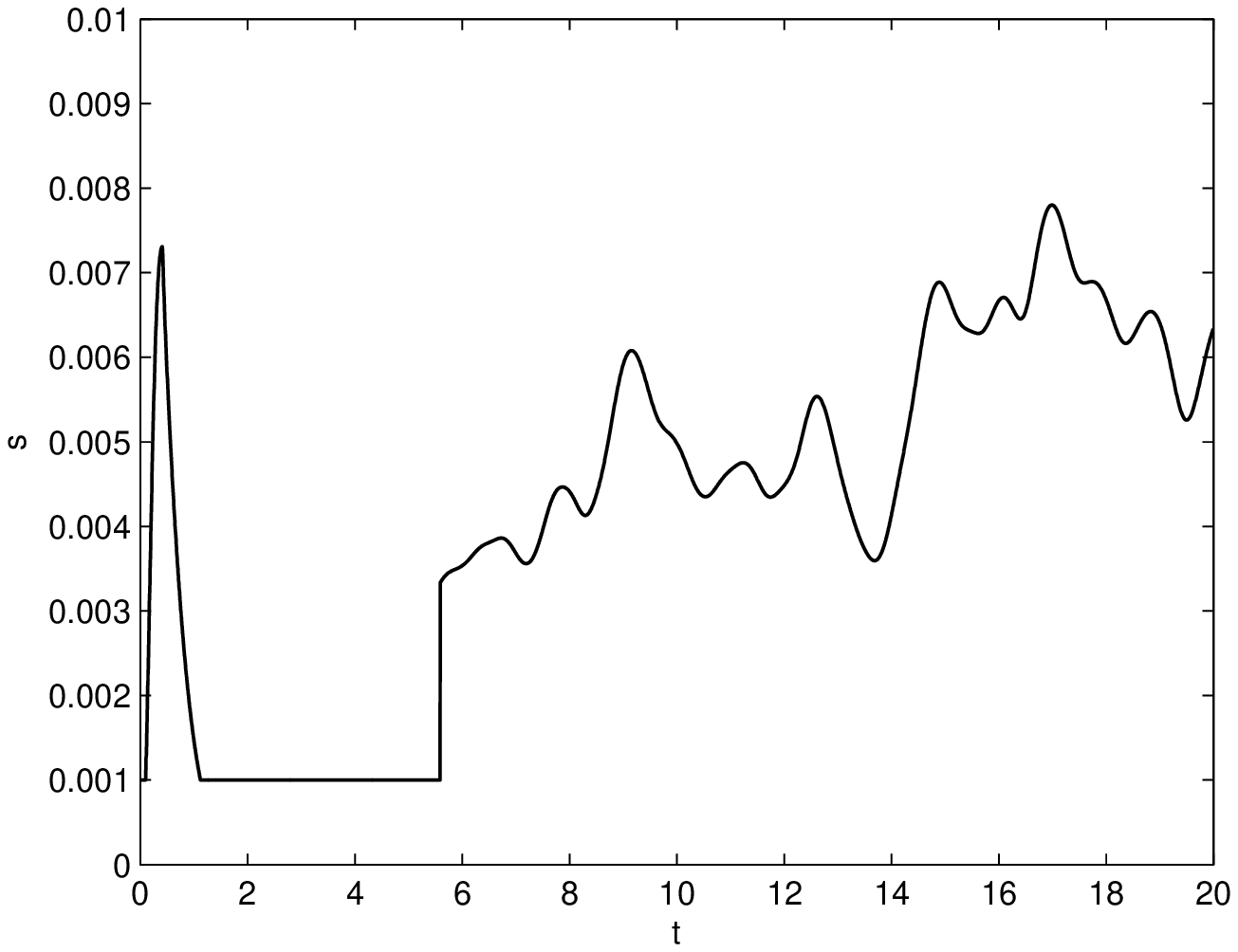}
\caption{Adaptive time step evolution.}
\label{sec5_fig2_adaptive}
\end{figure}

The corresponding time step evolution is presented in Figure \ref{sec5_fig2_adaptive}.
We see that the time step turns larger when $|U'|$ becomes smaller, while the time step turns smaller when the energy decreases quickly.
After the sharp energy decay, at about $t=6$, the parameter $\alpha$ is set to be $100$, so the time step enlarges rapidly.
After $t=6$, the time step increases overall except some fluctuations because of the alteration of the sign of $U''$.
The time step approaches the maximum $s_{\max}$ when the evolution goes on for sufficiently long time.

\begin{table}[!htp]
\centering
\caption{CPU time (seconds) comparison.}
\label{sec5_tab1}
\begin{tabular}{|c|c|c|c|c|c|c|}
\hline
$t$ & $1$ & $5$ & $10$ & $15$ & $20$ \\
\hline
constant & 1205 & 12334 & 36124 & 61799 & 86681 \\
\hline
adaptive & 1023 & 11473 & 23798 & 35501 & 45464 \\
\hline
\end{tabular}
\end{table}

Table \ref{sec5_tab1} lists the CPU time cost by using both constant and adaptive time stepping approaches.
At the early period of the computation, both approaches cost similar CPU time because the solution transits quickly and the small time steps are required.
Later, as the solution transits more and more slowly, the adaptive time stepping approach loosens the restriction for the time steps,
and as a result, the CPU time becomes significantly less in the comparison.
Though the last time step does not reach the maximum $s_{\max}$ in the considered interval $[0,20]$, as shown in Figure \ref{sec5_fig2_adaptive},
the CPU time of the adaptive time stepping approach is almost half of that of the constant case.

Here we see that the simulation with a larger time step performs well in the interval
where the energy decreases gently because of the unconditional energy stability of the difference scheme (\ref{sec3_full_discrete_scheme_1})-(\ref{sec3_full_discrete_scheme_2}),
while the linear scheme developed in \cite{LiXiao2014} will give wrong results with a step $s>0.003$
because the iterative method used to solve the scheme does not converge at some iterative step.
The adaptive time stepping approach is efficient for the simulation
because the moment at which the energy decays rapidly can be captured accurately and the large steps can be used without breaking down the iterations.

\subsection{The stochastic case}

The scheme (\ref{sec3_full_discrete_scheme_1_stoc}),(\ref{sec3_full_discrete_scheme_2}) with the uniform time step $s=0.001$ is used
to simulate the MMC-TDGL equation in the stochastic case.

The first experiment is to show the energy evolution though we do not obtain a theoretical result on it.
The noise strength is set to be $10^{-4}$ and the initial condition is $\phi(\br,0)=0.6+\zeta(\br)$,
where $\zeta(\br)$ is still a random disturbance as above, which implies that $\phi(\br,0)$ corresponds to an unstable state.
We will concentrate the energy evolution in the time interval $[0,10]$.
Because of the existence of the stochastic term,
we simulate $100$ independent samples to see the mean energy evolution, which is shown in Figure \ref{sec5_fig1_stoc}.
The bold black line represents the mean value of the $100$ samples and the thin blue line represents one of them.
It is found that the energy begins to decrease rapidly at about $t=2$,
which suggests the phase transition happens and the energy decays at each step after this moment.
As the noise is dominant before $t=2$ and the geometric evolution dominates after the occurrence of the phase transition,
we just need to pay more attention to some subintervals of $[0,2]$.
It is observed from Figure \ref{sec5_fig1_stoc}(b), (c) and (d) that
the mean energy does decrease throughout, although the energy of one of the samples may increase at some moments.

\begin{figure}[h]
\centering
\subfigure[$0\le t\le 10$]{\includegraphics[scale=0.45]{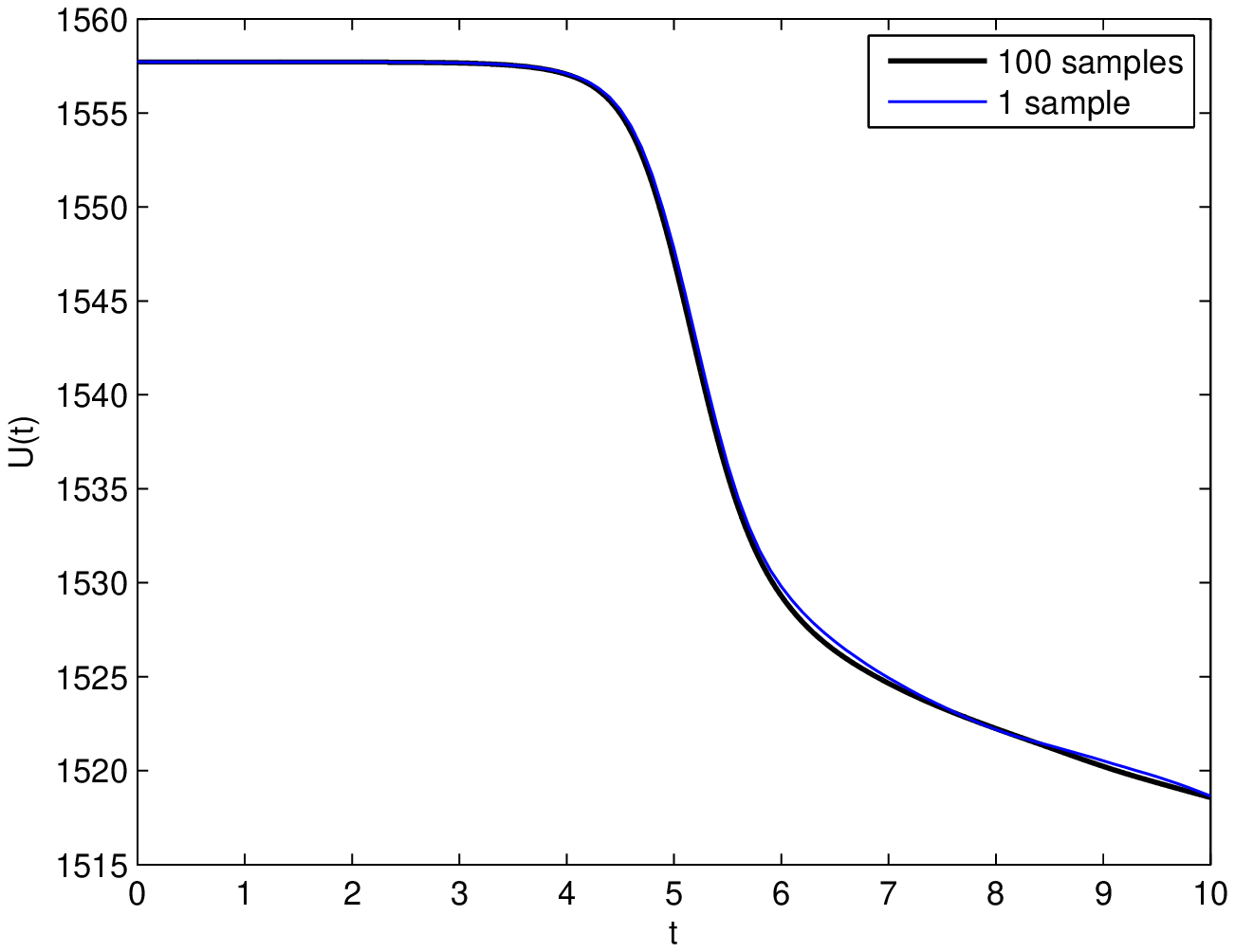}}
\subfigure[$0.19\le t\le 0.24$]{\includegraphics[scale=0.45]{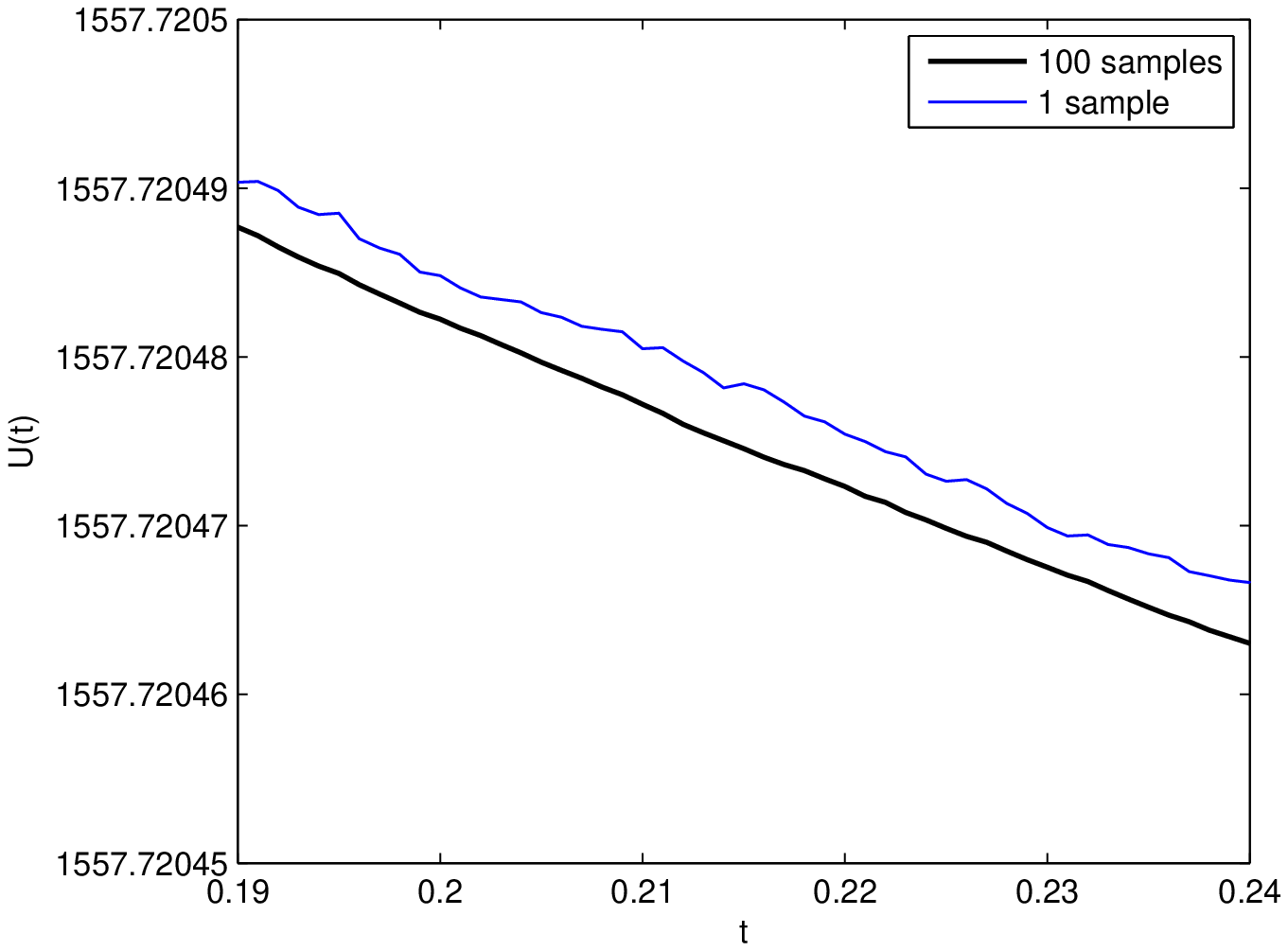}}
\subfigure[$0.32\le t\le 0.37$]{\includegraphics[scale=0.45]{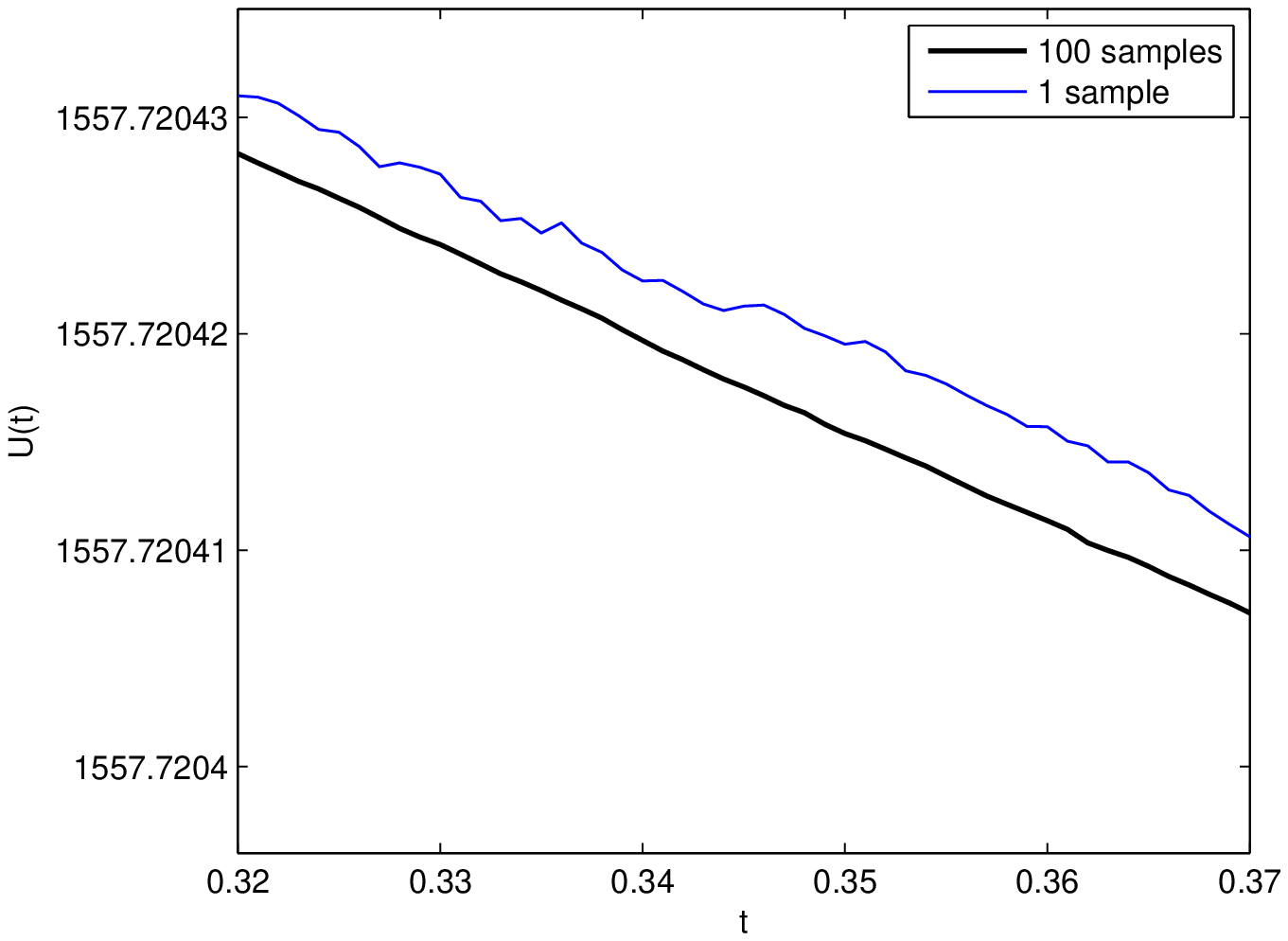}}
\subfigure[$0.5\le t\le 0.55$]{\includegraphics[scale=0.45]{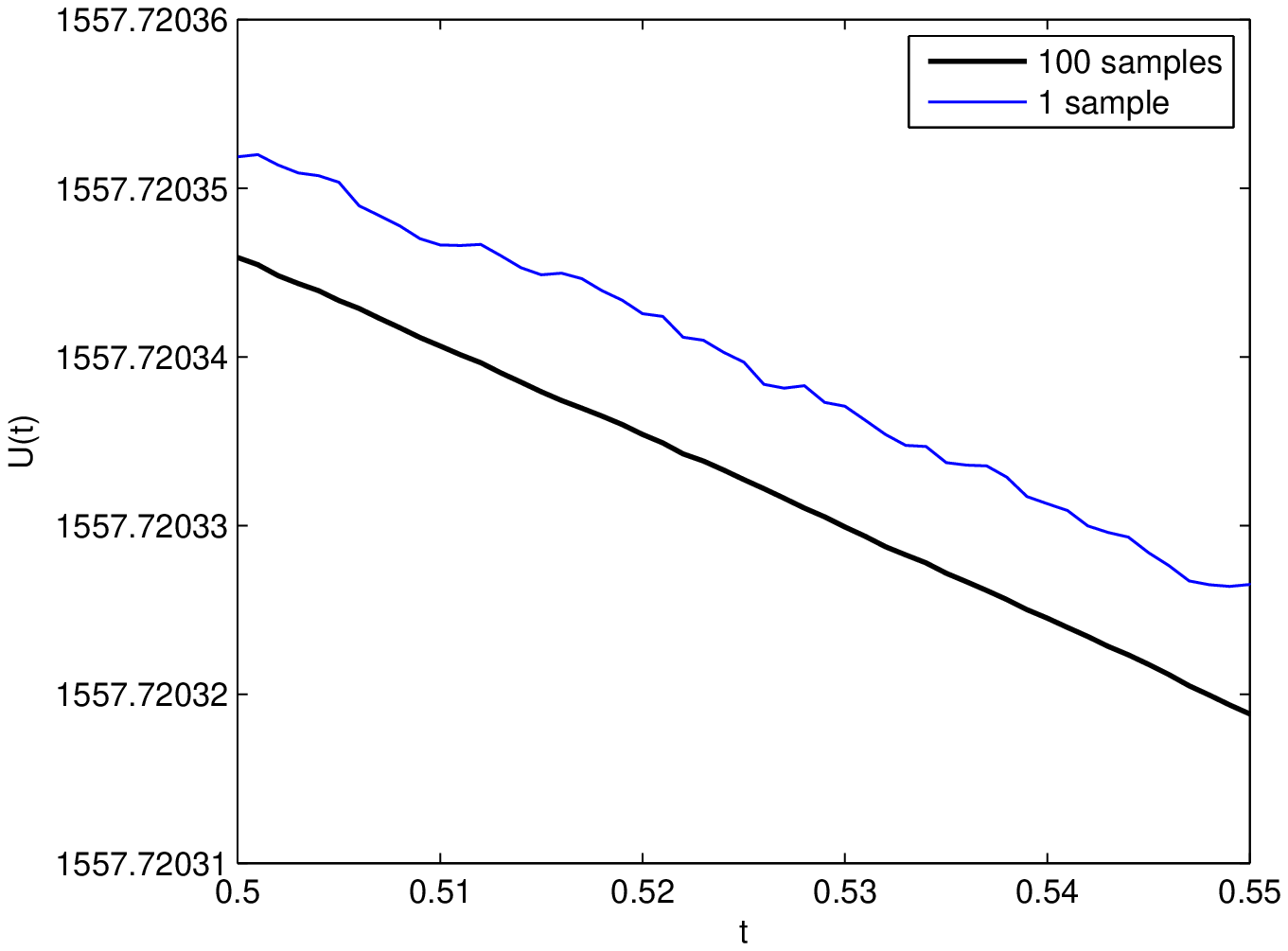}}
\caption{Energy evolution for the results obtained by the scheme
{\rm(\ref{sec3_full_discrete_scheme_1_stoc}),(\ref{sec3_full_discrete_scheme_2})}.}
\label{sec5_fig1_stoc}
\end{figure}

The second experiment is to compare the numerical solutions with those presented in the existing work \cite{LiXiao2014}.
The noise strength is set to be $10^{-4}$ and $10^{-3}$, respectively,
the Huggins parameter is $\chi=1.975$ and the initial state is uniform, namely, $\phi(\br,0)\equiv 0.3$.
We will concentrate the structures of the numerical solutions $\phi(\br,t)$ at $t=8$, $t=13$, $t=19$ and $t=25$.
Figure \ref{sec5_fig3_stoc} presents the results
corresponding to the cases $\varepsilon=10^{-4}$ and $\varepsilon=10^{-3}$.
Because of the existence of the stochastic term, these structures cannot be identical completely to the figures presented in \cite{LiXiao2014}.
Nevertheless, we still observe the same transition processes of the phase structures by the comparison.

\begin{figure}[!htp]
\centering
\subfigure[$t=8$]{\includegraphics[scale=0.35]{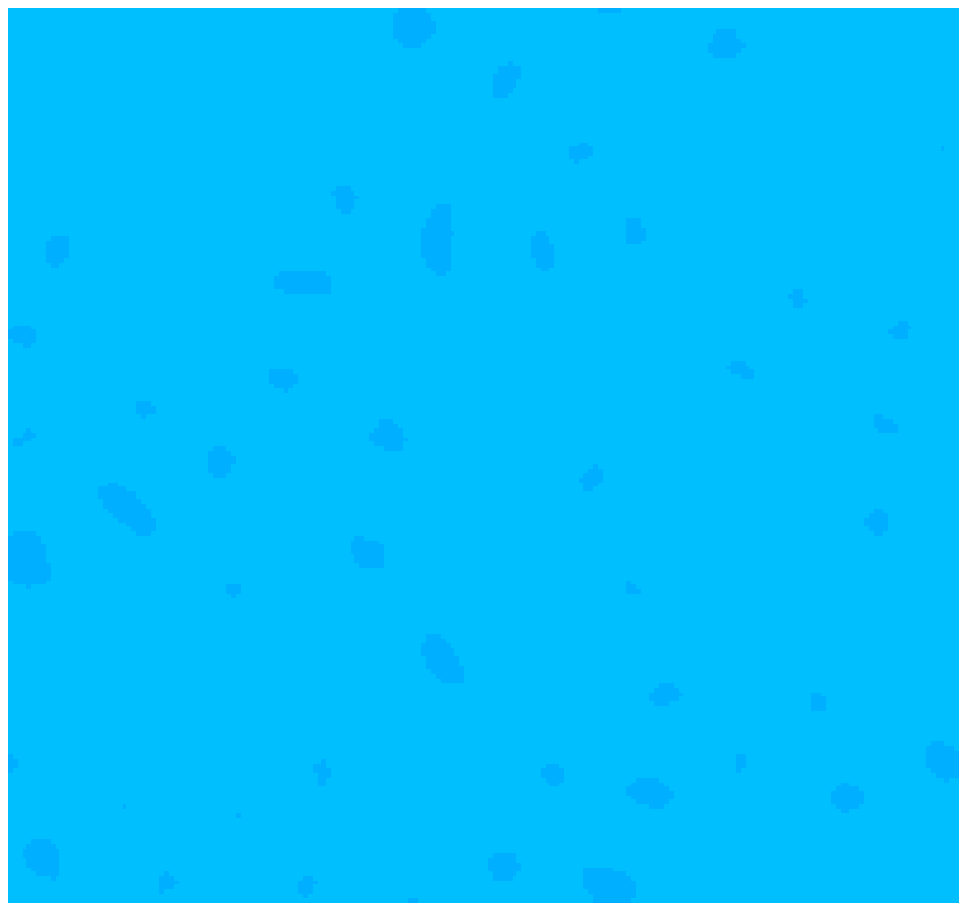}}
\subfigure[$t=13$]{\includegraphics[scale=0.35]{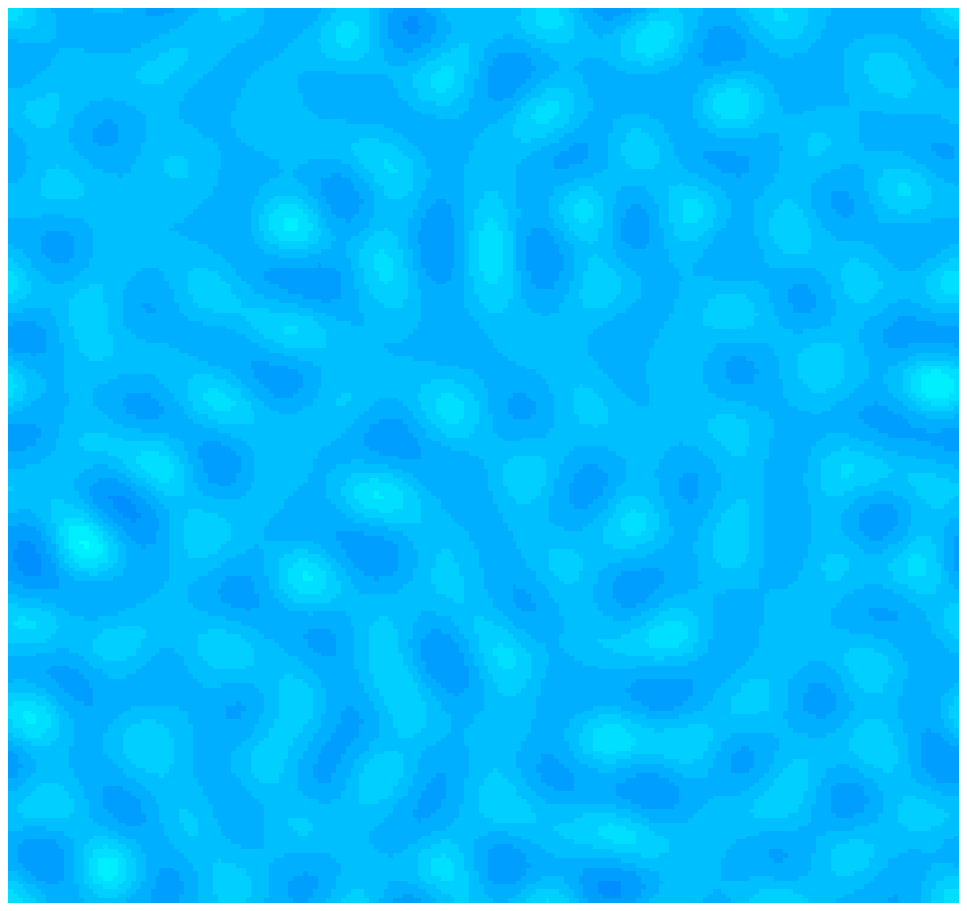}}
\subfigure[$t=19$]{\includegraphics[scale=0.35]{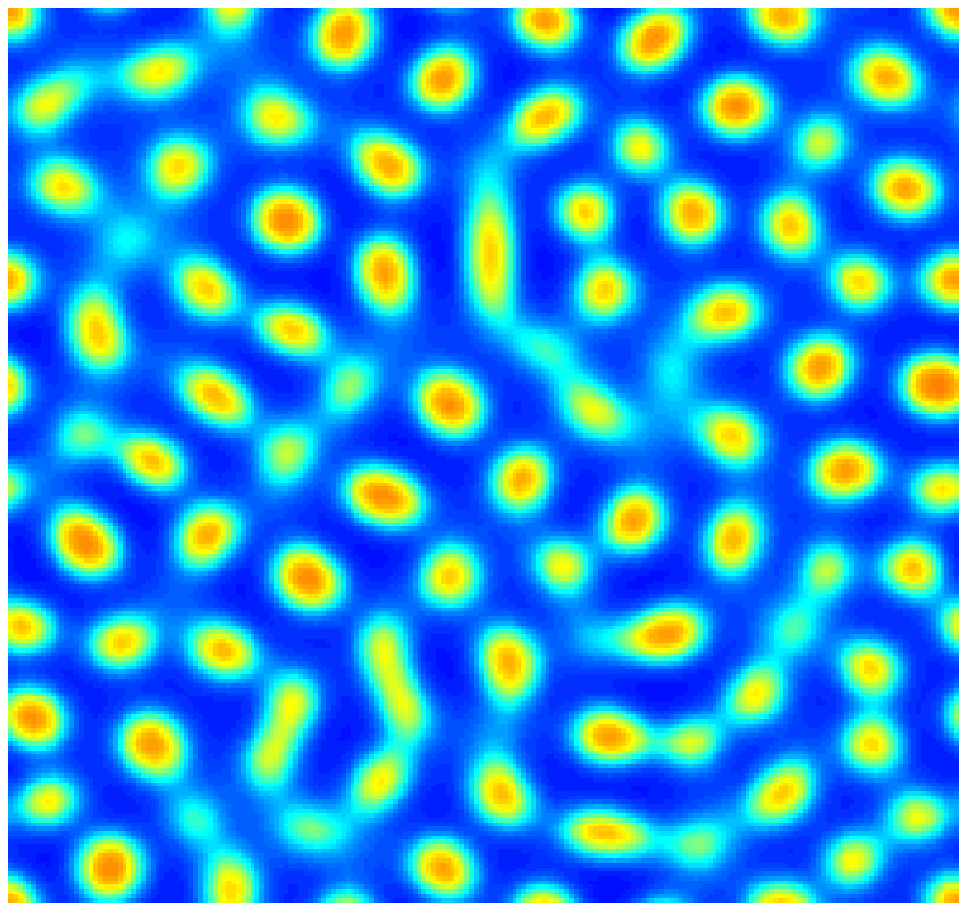}}
\subfigure[$t=25$]{\includegraphics[scale=0.35]{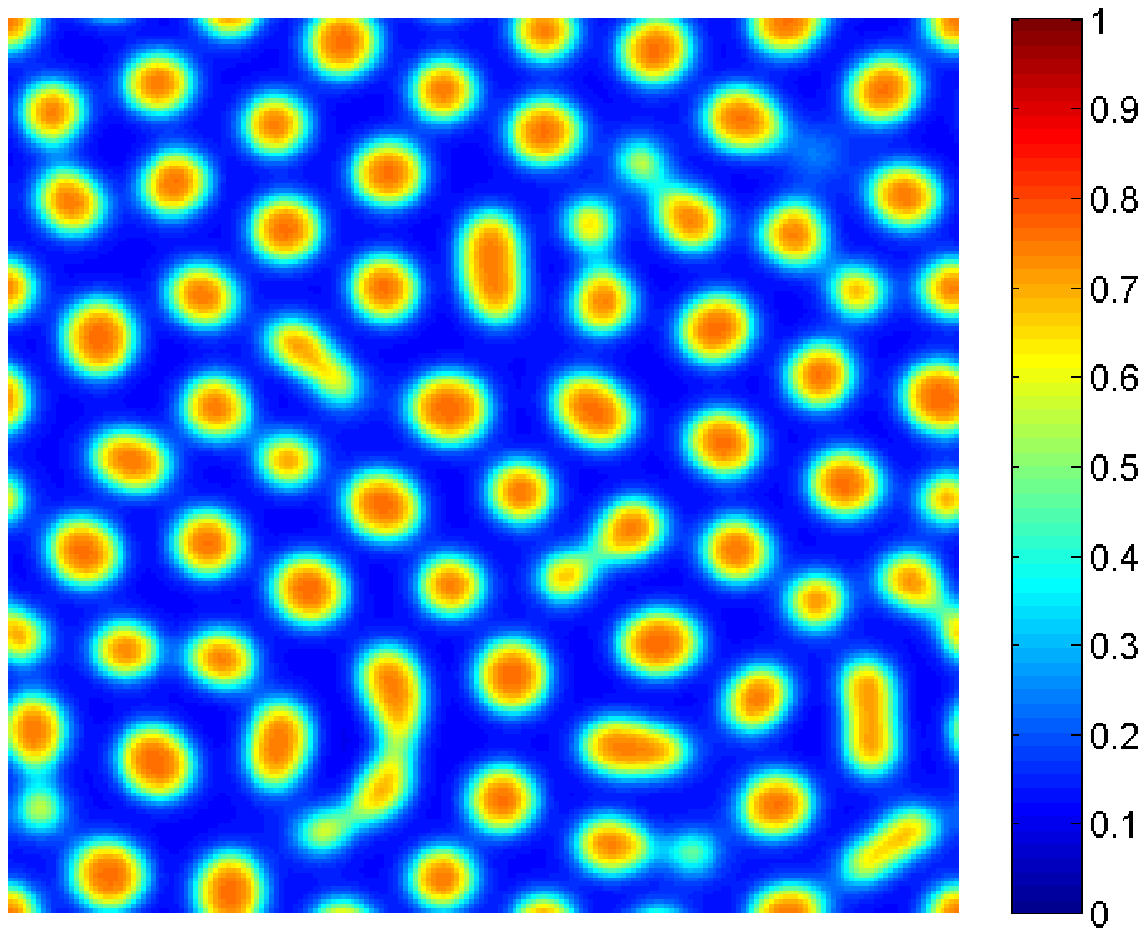}}\\
\subfigure[$t=8$]{\includegraphics[scale=0.35]{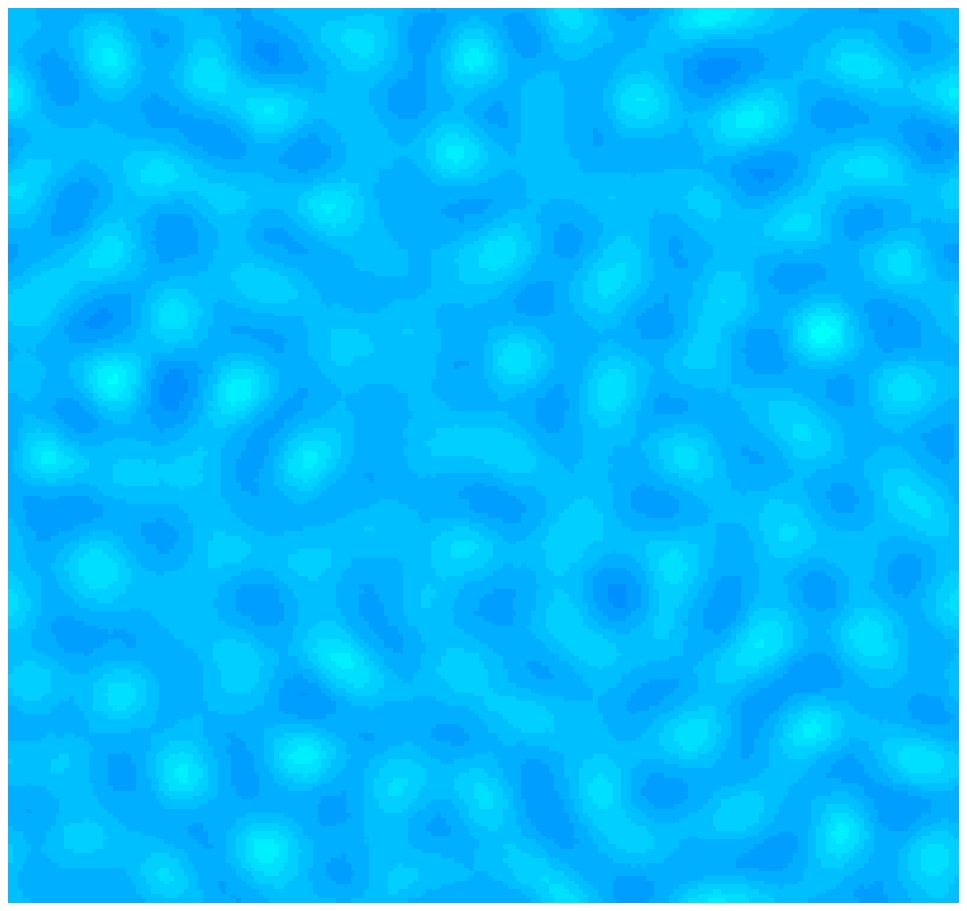}}
\subfigure[$t=13$]{\includegraphics[scale=0.35]{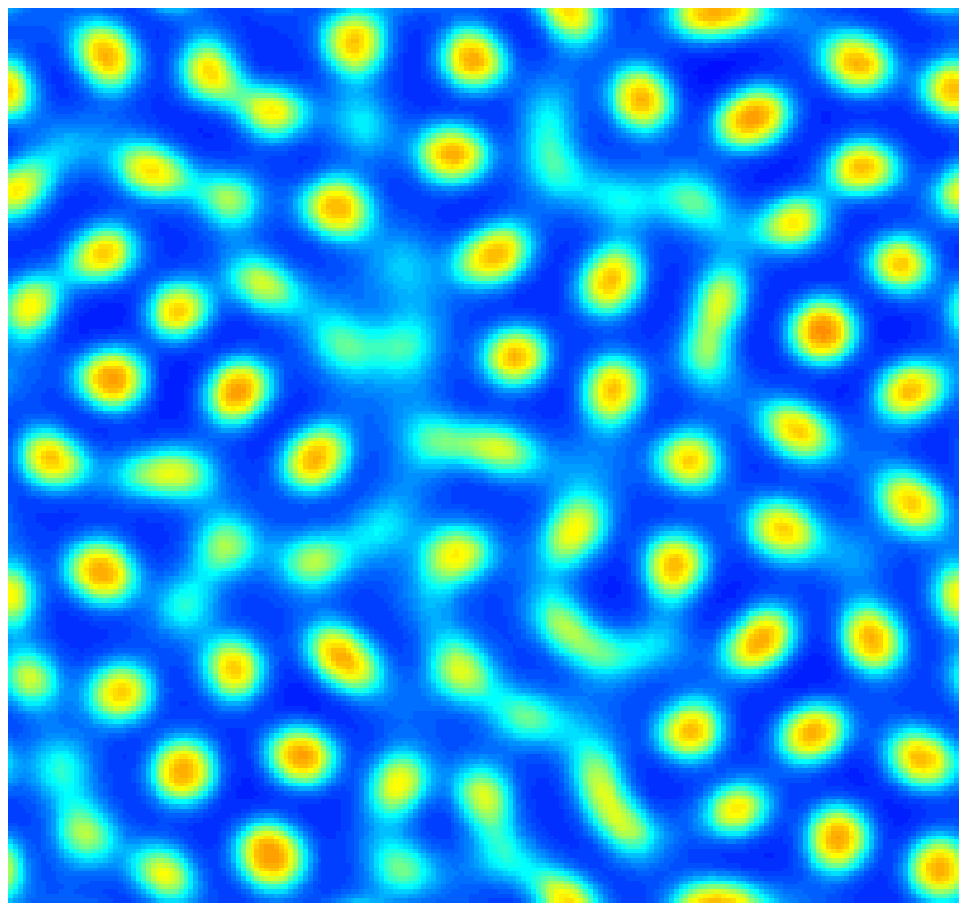}}
\subfigure[$t=19$]{\includegraphics[scale=0.35]{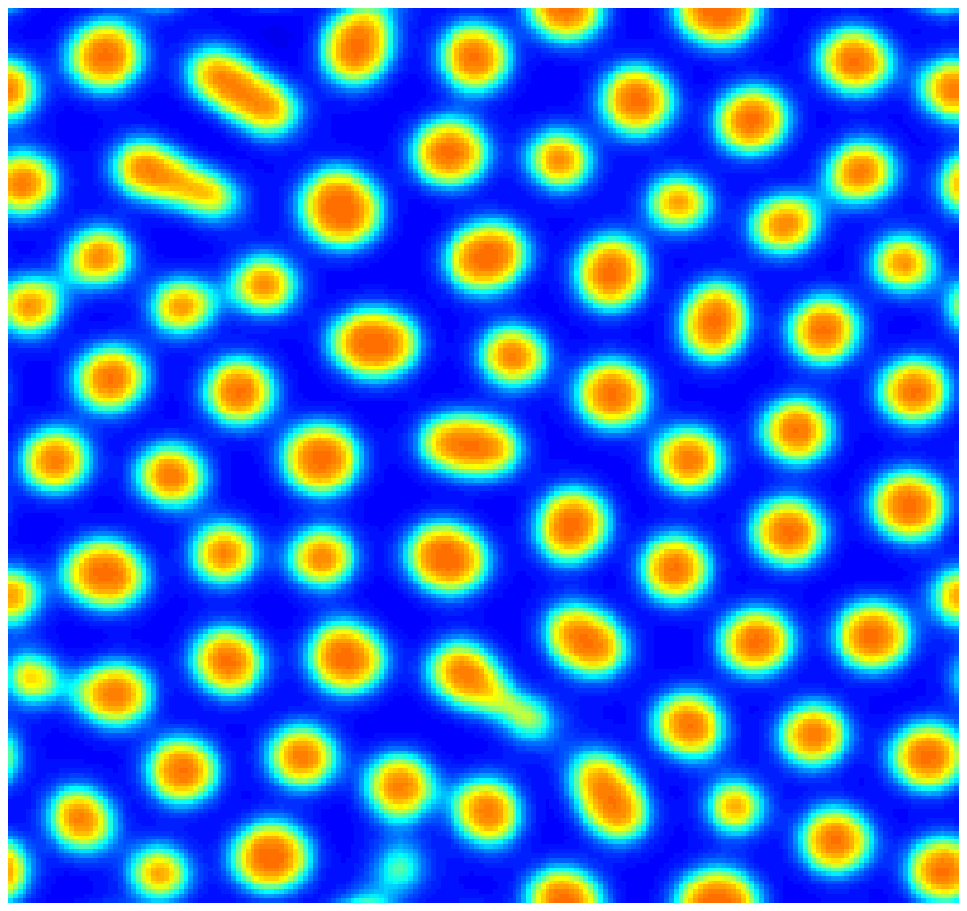}}
\subfigure[$t=25$]{\includegraphics[scale=0.35]{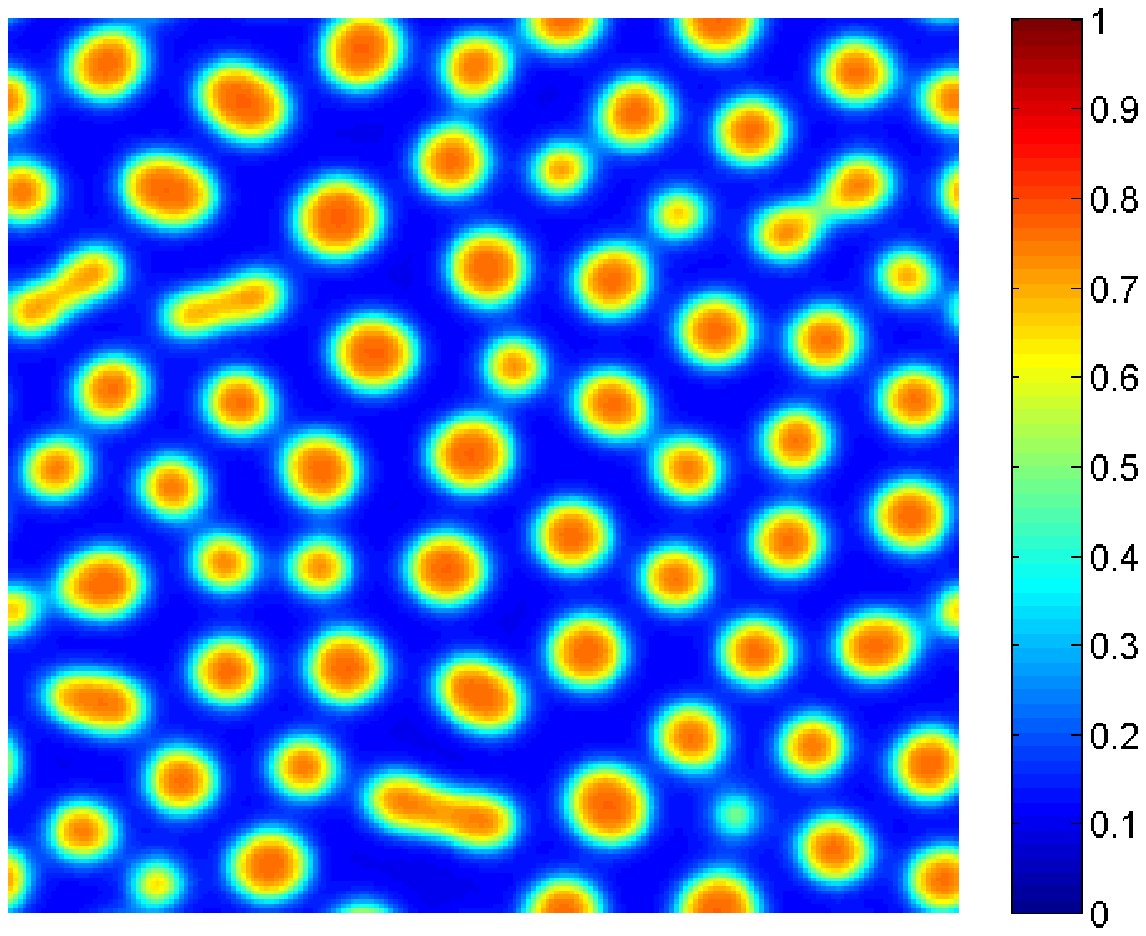}}
\caption{Evolution for $\varepsilon=10^{-4}$ (up) and $\varepsilon=10^{-3}$ (below).}\label{sec5_fig3_stoc}
\end{figure}

Since the initial uniform state is metastable, the system must climb over an energy barrier so that the phase transition occurs,
which is possible as long as the noise is sufficiently strong.
Obviously, the noise with strength either $10^{-3}$ or $10^{-4}$ is able to lead to the phase transition.
It is observed from Figure \ref{sec5_fig4_stoc} that the phase transition process relates to the noise strength $\varepsilon$.
We find from Figure \ref{sec5_fig4_stoc}(a) and (b) that the larger $\varepsilon$ is, the faster the phase transition occurs,
which is the same phenomenon as we have seen in our previous work \cite{LiXiao2014}.
Furthermore, Figure \ref{sec5_fig4_stoc}(c) and (d) show the energy curves in a short interval from the beginning
for the cases $\varepsilon=10^{-3}$ and $\varepsilon=10^{-4}$, respectively,
and the energy barrier on the transition path is observed.
More precisely, we find that both the energy curves increase in a very short time interval,
and then keep the decay trend with tiny oscillations till the end,
which is consistent with the phenomena illustrated in \cite{ZhangWei2012}.
In addition,
if we denote by $\Delta U(\varepsilon)$ the difference between the initial and maximum energy under the noise with strength $\varepsilon$,
we find from the labels of $y$-axis in (c) and (d) that $\Delta U(10^{-3})\approx100\Delta U(10^{-4})$.
It is suggested that the state that the system reaches with $\varepsilon=10^{-3}$ is more ``unstable'' than that with $\varepsilon=10^{-4}$,
which is why the phase transition is faster for the case $\varepsilon=10^{-3}$ and slower for the case $\varepsilon=10^{-4}$.

\begin{figure}[h]
\centering
\subfigure[$0\le t\le 100$]{\includegraphics[scale=0.45]{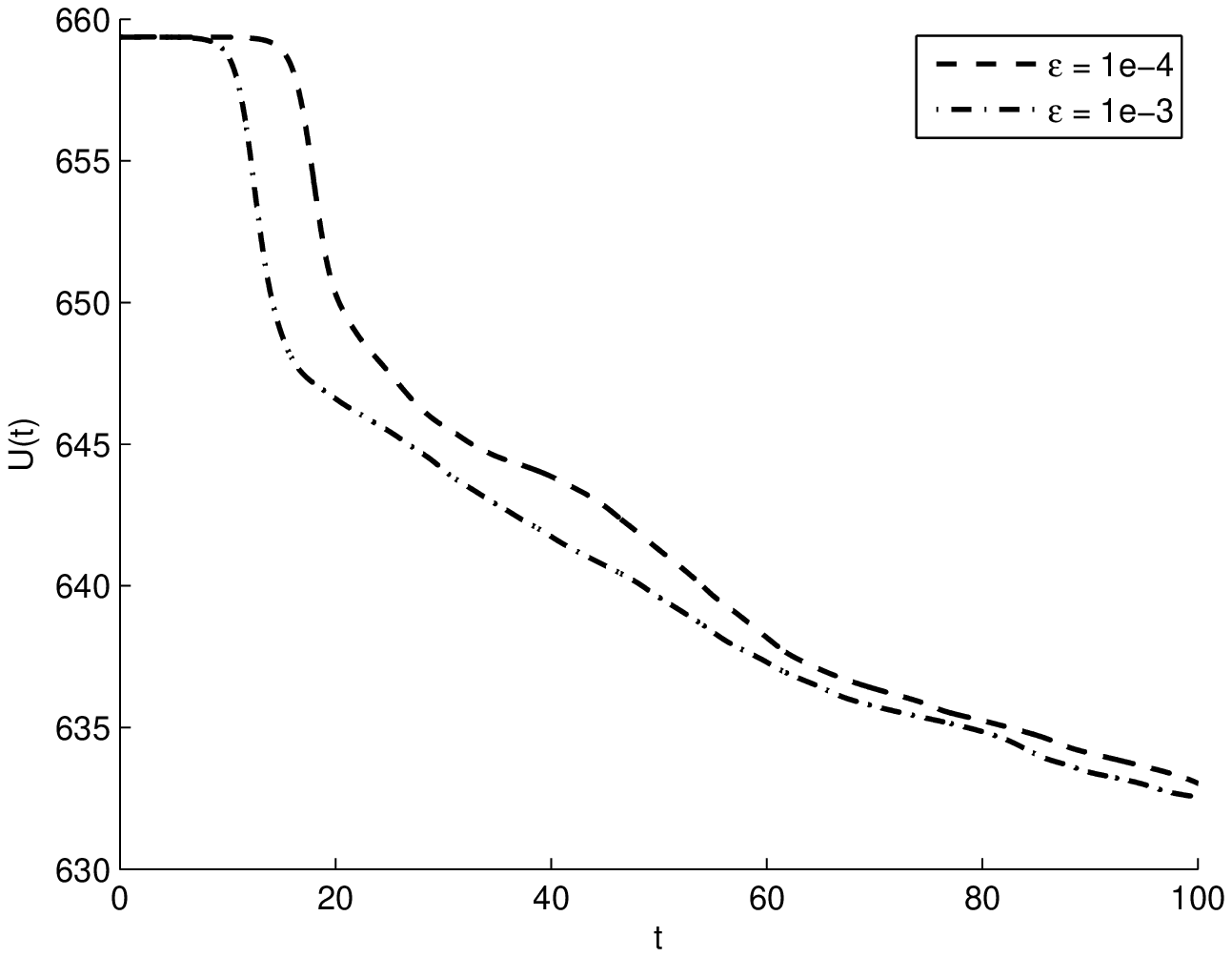}}
\subfigure[$0\le t\le 30$]{\includegraphics[scale=0.45]{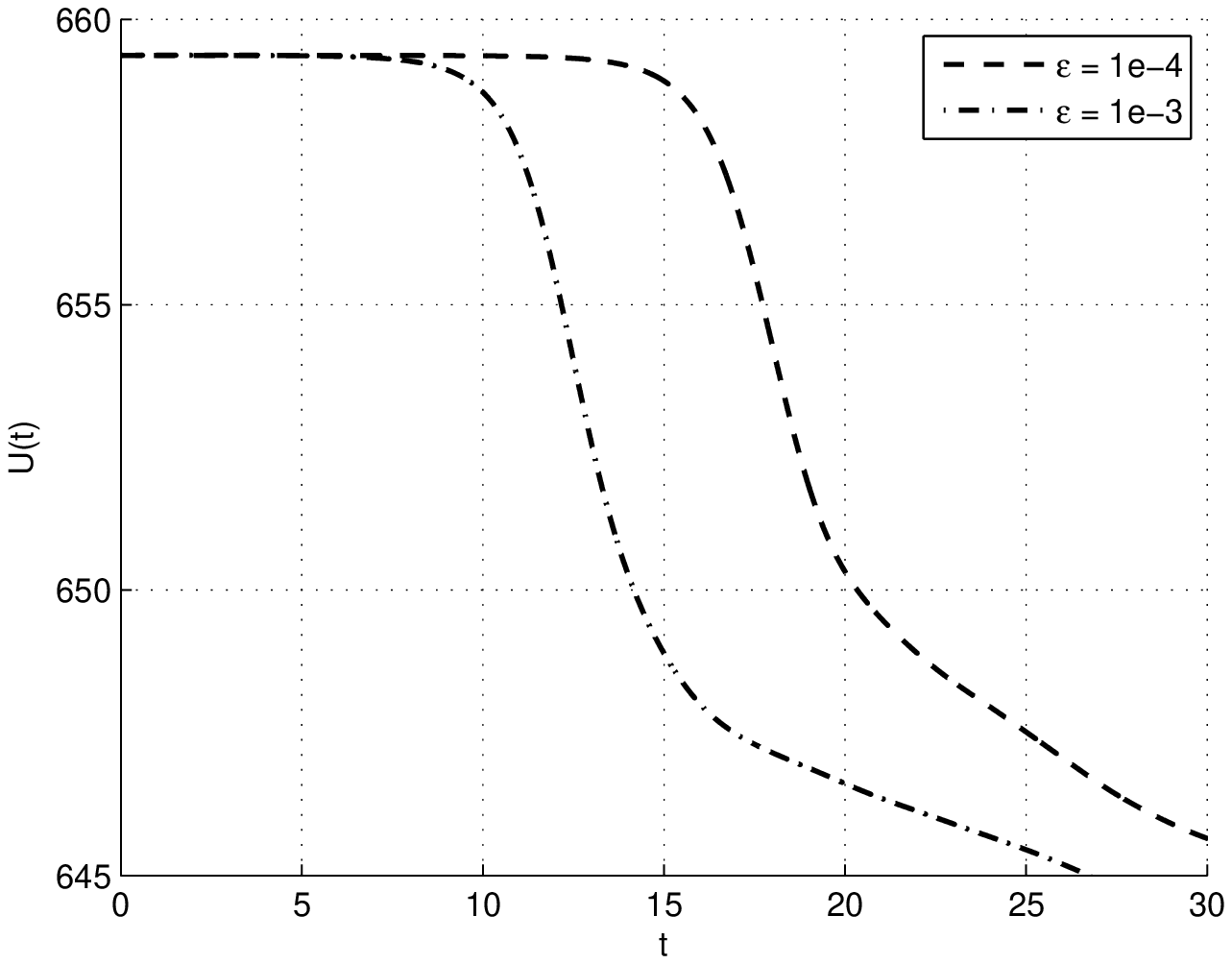}}
\subfigure[$0\le t\le 5$, $\varepsilon=10^{-3}$]{\includegraphics[scale=0.45]{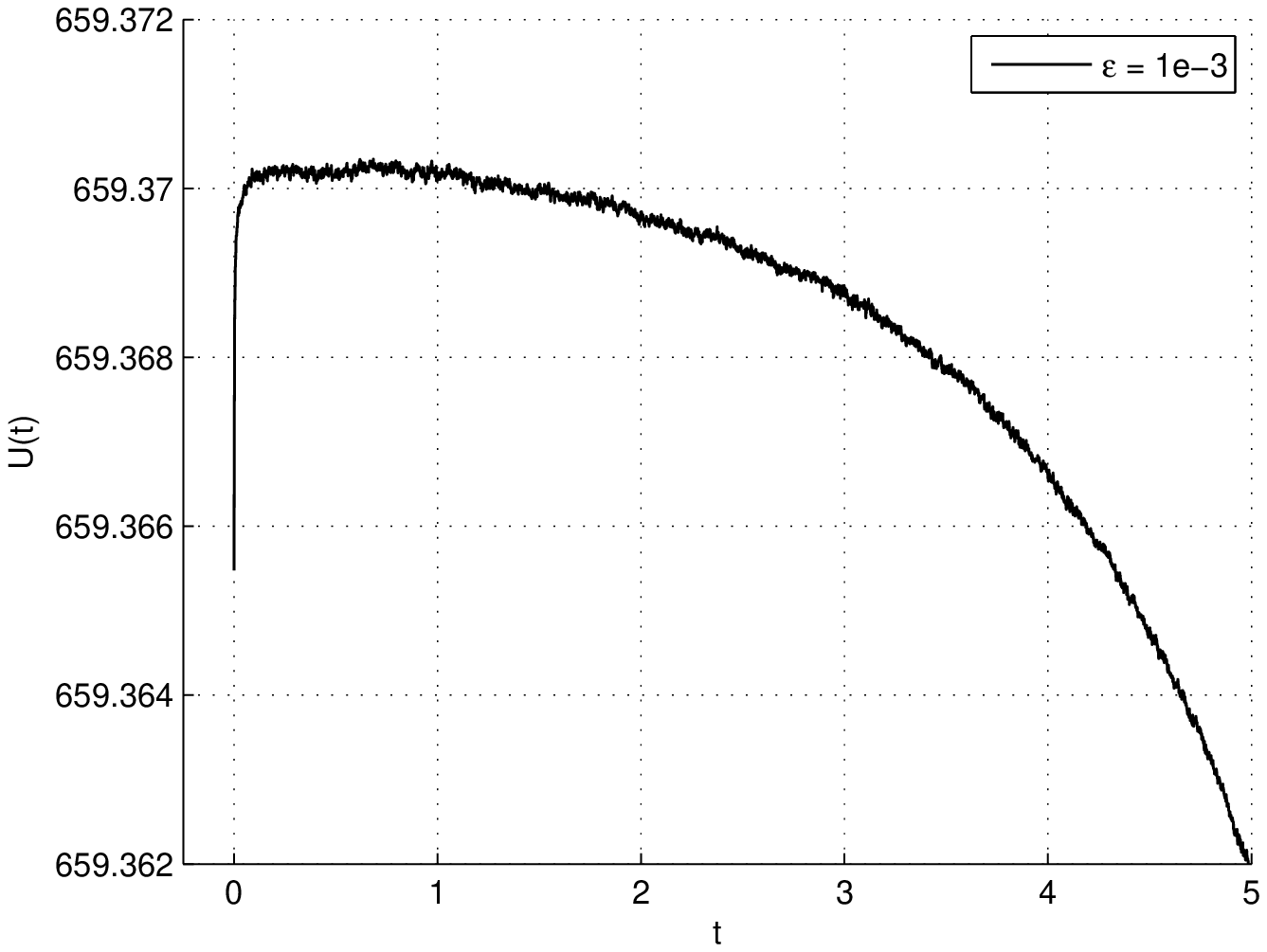}}
\subfigure[$0\le t\le 5$, $\varepsilon=10^{-4}$]{\includegraphics[scale=0.45]{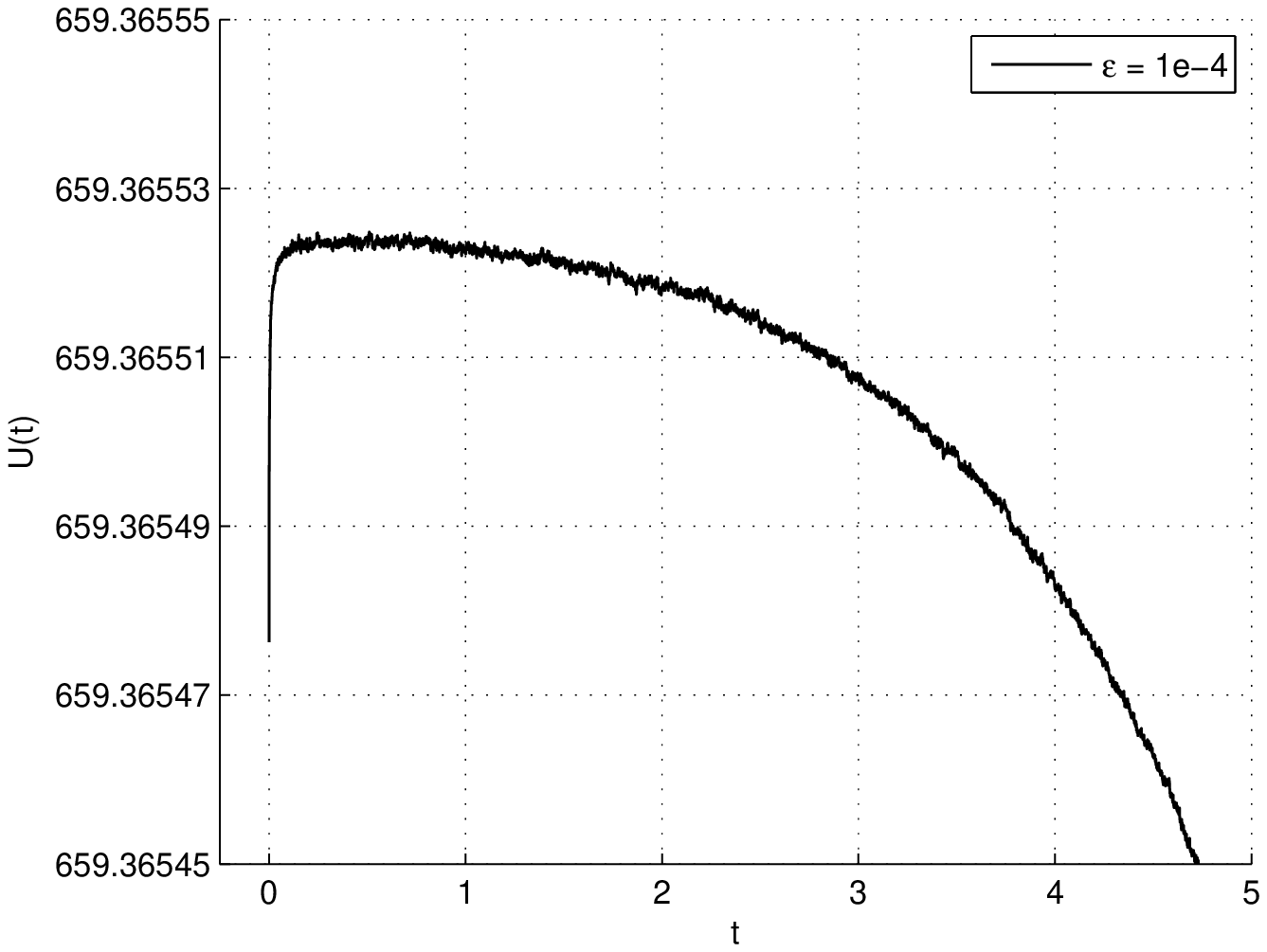}}
\caption{Energy evolution for $\varepsilon=10^{-4}$ and $10^{-3}$.}
\label{sec5_fig4_stoc}
\end{figure}

\section{Conclusions}

In this paper, we develop an unconditionally energy stable difference scheme for
a Cahn-Hilliard equation, i.e., the MMC-TDGL equation in the non-stochastic case.
The method is based on a convex splitting of the energy functional, namely, a sum of convex and concave parts.
The unconditional energy stability of the difference scheme for the non-stochastic case is obtained via a fundamental estimate of the convex splitting.
The unconditional unique solvability is proved by constructing a convex functional whose unique minimum is rightly the solution of the scheme.
Based on such fact, an optimization algorithm, the Newton method, is used to solve the difference scheme.
For the MMC-TDGL equation in the stochastic case,
we develop the difference scheme by using the same technique as the non-stochastic case without proving the energy law.
We sample one hundred trajectories with the same initial unstable state and observe that the mean of their energies decreases strictly.
The results of the stochastic simulations are consistent with the existing work \cite{LiXiao2014}
and the energy barrier on the transition path is observed.

Since the difference scheme for the non-stochastic is unconditionally stable,
we consider the adaptive time stepping techniques to accelerate the long time simulation.
Based on the energy variation, we use an adaptive time stepping formula (\ref{sec4_adapted_time}),
where we consider $\alpha$ as a function with respect to the second-order derivative $U''$ instead of the constant in \cite{QiaoZhonghua2011}.
The basic idea is that we desire a large $\alpha$ when $U''$ changes from positive to negative
in order to capture the moment when the energy begins to decay rapidly.
Numerical experiments show the efficiency of the adaptive time stepping method,
i.e., using the large time steps for computation without losing accuracy.
The adaptive time stepping strategy developed in this paper is appropriate to other problems with both flat and sharp decay of the energy.

According to the framework of the convex splitting method in \cite{Wise2009},
the difference scheme (\ref{sec3_full_discrete_scheme_1})-(\ref{sec3_full_discrete_scheme_2}) should be first-order convergent in time.
One of the future works in this direction is to carry out more rigorous analysis
for some error estimates for the scheme (\ref{sec3_full_discrete_scheme_1})-(\ref{sec3_full_discrete_scheme_2}).
Obtaining a satisfactory error estimate for a numerical scheme for the MMC-TDGL equation seems difficult,
since the free energy functional presents high nonlinearity and singularity in both bulk and interface parts
so that the common strategies, such as the technique used in \cite{Wise2009}, will not work.
We need to find some proper functions or appropriate upper bounds to control the singular parts.
Under the framework given in \cite{Wise2012}, second-order (in time) convex splitting schemes for the MMC-TDGL equation may be constructed.
Still, it is difficult to obtain the rigorous error estimates.
We will put efforts on error estimates on energy stable schemes for the MMC-TDGL equation in the future work.
Other future works in this direction include the linear difference schemes with the unconditional energy stability,
which is important in improving the efficiency of the simulations for the complicated nonlinear equations.

\Acknowledgements{
We would like to express our thanks to the anonymous reviewers
whose valuable comments and suggestions help us improve this article greatly.
We are grateful to Prof.~Du Qiang of Beijing Computational Science Research Center
and Prof. Ji Guanghua of Beijing Normal University for many valuable comments.
Li X thanks the Hong Kong Polytechnic University for support during his visit.
The research of Qiao Z H is partially supported by the Hong Kong Research Council GRF
grants 202112, 15302214, 509213 and NSFC/RGC Joint Research Scheme N\_HKBU204/12.
Zhang H is partially supported by NSFC/RGC Joint Research Scheme No.~11261160486,
NSFC grant No.~11471046 and the Ministry of Education Program for New Century Excellent Talents Project NCET-12-0053.}


\end{document}